\documentclass[11pt,sumlimits,nointlimits,namelimits,reqno]{amsart}
\usepackage{amsopn, amsfonts,amssymb,amscd,graphicx,tabularx,oldgerm,enumitem}
\usepackage[mathscr]{eucal}
\usepackage{color}
\usepackage[all]{xy}
\usepackage{comment}

\usepackage{amsthm}

\newtheorem {thm}{Theorem}[section]
\newtheorem {prop}[thm]{Proposition}
\newtheorem {conj}[thm]{Conjecture}
\newtheorem{cor}[thm]{Corollary}
\newtheorem {lem}[thm]{Lemma}
\newtheorem*{theorem*}{Theorem}

\newtheorem {question}{Question}

\theoremstyle{definition}
\newtheorem{defn}[thm]{Definition}
\newtheorem{ex}[thm]{Example}

\theoremstyle{remark}
  \newtheorem{remark}[thm]{Remark}

\numberwithin{equation}{section}

\newcommand{\K}{{\mathbb K}}
\newcommand{\T}{{\mathbb T}}
\newcommand{\Q}{{\mathbb Q}}
\newcommand{\Z}{{\mathbb Z}}
\newcommand{\C}{{\mathbb C}}
\newcommand{\R}{{\mathbb R}}

\newcommand{\F}{{\mathbb F}}
\newcommand{\calK}{{\mathcal K}}
\newcommand{\tK}{{ \widetilde{\mathcal K}}}
\newcommand{\calA}{{\mathcal A}}
\newcommand{\calH}{{\mathcal H}}
\newcommand{\calQ}{{\mathcal Q}}
\newcommand{\calI}{{\mathcal I}}
\newcommand{\calS}{{\mathcal S}}

\newcommand{\Fq}{{\mathbb F_{q}}}
\newcommand{\Fp}{{\mathbb F_{p}}}

\newcommand{\GN}{{\mathbb G_{N}}}
\newcommand{\GM}{{\mathbb G_{M}}}
\newcommand{\GL}{{\mathbb G_{L}}}
\newcommand{\G}{{\mathbb G}}

\renewcommand{\mod}{\textup{mod}\hspace{0.5mm}}

\newcommand{\ord}{\textup{ord}\,}
\newcommand{\res}{\textup{res}}

\newcommand{\usum}{\displaystyle \sum}

\newcommand{\mmod}[1]{\,\,(\text{mod}\,\,#1)}

\begin{document}

\title [Equidistribution of polynomial sequences in function fields]
{{\bf Equidistribution of polynomial sequences in\\ 
function fields, with applications}}

\author{Th\'ai Ho\`ang L\^e}
\address{T. H. L\^e, Department of Mathematics,
University of Mississippi,
305 Hume Hall,
University, MS 38677, USA}
\email{leth@olemiss.edu}

\author{Yu-Ru Liu}
\address{Y.-R. Liu, Department of Pure Mathematics,
University of Waterloo,
200 University Avenue West, 
Waterloo, ON, N2L 3G1, Canada }
 \email{yrliu@math.uwaterloo.ca}

\author{Trevor D. Wooley}
\address{T. D. Wooley, Department of Mathematics,
Purdue University,
150 N. University Street,
West Lafayette, IN 47907, USA}
\email{twooley@purdue.edu}

\thanks{The second author is supported by an NSERC discovery grant. The third author is 
supported by NSF grants DMS-1854398 and DMS-2001549.}


\begin{abstract}We prove a function field analog of Weyl's classical theorem on 
equidistribution of polynomial sequences. Our result covers the case in which the degree 
of the polynomial is greater than or equal to the characteristic of the field, which is a 
natural barrier when applying the Weyl differencing process to function fields. We also 
discuss applications to van der Corput, intersective and Glasner sets in function fields.
\end{abstract}

\maketitle

\section{Introduction} \label{sec:intro}
Equidistribution theory started with Weyl's seminal paper \cite{weyl2}. We recall that a 
sequence $(a_{n})_{n=1}^{\infty}$ of real numbers is said to be \textit{equidistributed} 
$(\mod{1})$ if for any interval $[\alpha,\beta] \subset [0,1)$, we have
\[
\lim_{N \rightarrow \infty}N^{-1}\text{card}\left\{n\in [1,N]\cap \Z^+: 
\{ a_n \}\in [\alpha, \beta]\right\}= \beta-\alpha .
\]
Here, we write $\Z^+$ for the set of positive integers and $\{a\}$ for the fractional part 
of a real number $a$, which is to say $a-\lfloor a\rfloor$, where $\lfloor a\rfloor$ denotes 
the largest integer not exceeding $a$. Write $e(x)=e^{2\pi ix}$. Then 
\textit{Weyl's criterion} asserts that the sequence $(a_{n})_{n=1}^{\infty}$ is 
equidistributed $(\mod 1)$ if and only if for any integer $m \neq 0$, we have
\[
\lim_{N \rightarrow \infty}\frac{1}{N}\biggl|\sum_{n=1}^N e(ma_{n})\biggr| =0.
\]

\par Let $f(u)=\sum_{r=0}^k\alpha_r u^r$ be a polynomial with real coefficients having 
degree $k$. Weyl made the important observation that by squaring the sum 
$\big| \sum_{n=1}^N e(f(n))\big|$, one can estimate it in terms of other exponential 
sums involving the shift $f(u+h)-f(u)$, which is, for each $h\in \Z^+$, a polynomial of 
degree $k-1$. This process is called \textit{Weyl differencing}. If one continues the 
differencing process, then the polynomial in question becomes linear after $k-1$ steps. 
Using this observation, Weyl \cite{weyl2} proved that the sequence 
$(f(n))_{n=1}^\infty$ is equidistributed $(\mod 1)$ if and only if at least one of the 
coefficients $\alpha_1,\ldots,\alpha_k$ of $f$ is irrational. The proof of this result was 
later simplified with the help of \textit{van der Corput's difference theorem} \cite{vdc}, 
which shows that, if for any $h\in \Z^+$ the sequence $(a_{n+h}-a_{n})_{n=1}^\infty$ 
is equidistributed $(\mod 1)$, then the sequence $(a_{n})_{n=1}^\infty$ is also 
equidistributed $(\mod 1)$. Using van der Corput's difference theorem, Weyl's 
equidistribution theorem for polynomials follows easily by induction on the degree of the 
polynomial. This remains to date the standard proof of Weyl's result.\par

Denote by $\F_q$ the finite field of $q$ elements whose characteristic is $p$ and let 
$\F_q[t]$ be the polynomial ring over $\F_q$. Since $\Z$ and $\F_q[t]$ share many 
similarities from analytic and number-theoretic points of view, it is natural to study 
equidistribution in the latter setting. Let $\K =\F_q(t)$ be the field of fractions of 
$\F_q[t]$. When $f/g\in \K$, with $f,g\in \F_q[t]$ and $g\ne 0$, we define a norm
$|f/g|=q^{\deg f-\deg g}$ (with the convention that $\deg 0=-\infty$). The completion of 
$\K$ with respect to this norm is $\K_\infty =\F_q((1/t))$, the field of formal Laurent 
series in $1/t$. In other words, every element $\alpha \in \K_\infty$ can be written in the 
form $\alpha =\sum_{i=-\infty}^n a_i t^{i}$ for some $n \in \Z$ and $a_i \in \F_q$ 
$(i\le n)$. Therefore, one sees that $\F_q[t]$, $\K$, $\K_\infty$ play the roles of $\Z$, 
$\Q$, $\R$, respectively. Let 
\[
\T= \Biggl\{ \sum_{i\le -1}a_i t^i: a_i \in \Fq\,\, (i \le -1)\Biggr\}.
\]
This compact group is the analog of the unit interval $[0,1)$. Let $\lambda$ be a 
normalized Haar measure on $\T$ such that $\lambda(\T)=1$. For $M \in \Z^+$, let 
$I=(c_1,\ldots,c_M)$ be a finite sequence of elements of $\F_q$. A set of the form 
\[
\mathcal{C}_I=\Biggl\{ \sum_{i\le -1}a_i t^i\in \T: \text{$a_i=c_{-i}$ 
$(-M\le i\le -1)$}\Biggr\}
\]
satisfies $\lambda(\mathcal{C}_I)=q^{-M}$. Thus, we refer to the set $\mathcal{C}_I$ 
as a \textit{cylinder set of radius} $q^{-M}$. The topology on $\T$ induced by the norm 
$|\cdot|$ is generated by cylinder sets. Therefore, cylinder sets play the role of intervals.

\par For $\alpha =\sum_{i=-\infty}^{n} a_it^{i}\in \K_\infty$ with $a_n\neq 0$, we 
define $\ord \alpha  =n$. Therefore, one has $|\alpha | = q^{\ord \alpha}$. We say 
$\alpha$ is \textit{rational} if $\alpha \in \K$ and \textit{irrational} if 
$\alpha \not \in \K$. We define $\{ \alpha \} =\sum_{i\le -1}a_i t^{i}\in \T$ to be the  
{\it fractional part} of $\alpha$, and we refer to $a_{-1}$ as the \textit{residue} of 
$\alpha$, denoted by $\text{res}\,\alpha$. Next we define the exponential function on 
$\K_\infty$. Let $\text{tr}:\F_q\rightarrow \F_p$ denote the familiar trace map given by
\[
\text{tr}(a)=a+a^p+a^{p^2}+\ldots +a^{p^{m-1}},
\]
in which we suppose that $q=p^m$. There is a non-trivial additive character 
$e_q:\F_q\rightarrow \C^\times$ defined for each $a\in \F_q$ by taking 
$e_q(a)=e(\text{tr}(a)/p)$. This character induces a map, which we also denote by 
$e(\cdot)$, from $\K_\infty$ to $\C^\times$ by defining, for each element 
$\alpha\in \K_\infty$, the value of $e(\alpha )$ to be $e_q(\text{res}\,\alpha )$. For 
$N\in \Z^+$, we write $\G_N$ for the set of all polynomials in $\F_q[t]$ having degree 
smaller than $N$. The following notion of equidistribution was first introduced by Carlitz in 
\cite{carlitz} (see also \cite[Chapter 5, Section 3]{kn}).

\begin{defn}\label{definition1.1}
Let $(a_x)_{x \in \Fq[t]}$ be a sequence indexed by $\Fq[t]$ and taking values in 
$\K_{\infty}$. We say that $(a_x)_{x \in \Fq[t]}$ is {\it equidistributed} in $\T$ if for any 
cylinder set $\mathcal{C}\subset\T$, we have
\[
\lim_{N \rightarrow \infty} q^{-N}\text{card} \left\{ x \in \G_N :\{a_x\} \in \mathcal{C} 
\right\}= \lambda(\mathcal{C}).
\]
\end{defn}

Since one can prove analogs of Weyl's criterion and van der Corput's difference theorem in 
function fields, one expects to establish an $\F_q[t]$-analog of Weyl's equidistribution 
theorem for polynomial sequences. Let $f(u)=\sum_{r=0}^k \alpha_r u^r$ be a 
polynomial with coefficients in $\K_\infty$ having degree $k$. All earlier works on 
equidistribution in $\T$ have been restricted to the case in which $k<p$. Under this 
condition, Carlitz \cite{carlitz} proved an analog of Weyl's equidistribution theorem for the 
sequence $(f(x))_{x\in \F_q[t]}$. Dijksma \cite{dijksma} also established the same result 
for another stronger notion of equidistribution, subject to the same constraint $k<p$. In 
the work of both Carlitz and Dijksma, the use of Weyl differencing produces a factor of 
$k!$. When $k\ge p$, the latter factor is $0$, and hence this differencing method 
becomes ineffective in producing the desired equidistribution result. Actually, the following 
example, already known to Carlitz \cite[equation (6.8)]{carlitz}, shows that a direct 
$\Fq[t]$-analog of Weyl's equidistribution theorem is not always true when $k\ge p$. 

\begin{ex}\label{ex:1}
For $\alpha =\sum_{i=-\infty}^n a_i t^{i}\in \K_\infty$, define
\begin{equation}\label{eq:t}
T(\alpha)=a_{-1}t^{-1}+a_{-p-1}t^{-2}+a_{-2p-1}t^{-3}+\cdots . 
\end{equation}
Then $T$ is a linear map from $\K_\infty$ to $\T$ (this map will also be used in Section 
\ref{sec:equidistribution}). By setting $a_{-1}=a_{-p-1}=\ldots =0$, a countability argument 
shows that we can find an irrational element $\alpha\in \K_\infty$ with $T(\alpha)=0$. Given 
such an irrational element $\alpha$, it follows that for any element 
$x=\sum_{i=0}^{m} x_{i} t^{i}$ of $\F_q[t]$, the coefficient of $t^{-1}$ in 
$\alpha x^{p}$ is equal to
\[
a_{-1}x_0^{p} + a_{-p-1}x_1^p + a_{-2p-1}x_2^p+ \cdots =0,
\]
and thus the sequence $(\alpha x^p)_{x \in \F_q[t]}$ is not equidistributed in $\T$.
\end{ex}

It is desirable to give a complete description of all polynomials $f(u)\in \K_\infty [u]$ for 
which the sequence $(f(x))_{x \in \F_q[t]}$ is equidistributed in $\T$. However, in view of 
Example \ref{ex:1}, such a description may be complicated and not easy to state in such 
arithmetic terms as irrationality. In particular, equidistribution could fail if the degree of 
$f(u)$ is divisible by $p$. Furthermore, for a polynomial such as $\alpha x^p +\beta x$, it 
is not possible to determine whether or not one has equidistribution if one is equipped 
with information concerning $\alpha$ or $\beta$ alone, since the terms $x^p$ and $x$ 
``interfere" with one another, as the map $x \mapsto x^p$ is linear (see also 
\cite[equation (6.9)]{carlitz}). However, one may suspect that the only pathologies that 
prevent equidistribution are the ones described above (namely, exponents divisible by $p$ 
and interfering exponents). Thus one can make the following conjecture, which is the best 
possible insofar as irrationality hypotheses are imposed on a single coefficient. 

\begin{conj}\label{conj}
Let $\calK$ be a finite set of positive integers, suppose that $\alpha_r\in \K_\infty$ for 
$r\in \calK\cup \{0\}$, and define
\[
f(u)=\sum_{r \in \calK\cup\{0\}}\alpha_r u^{r}.
\]
Suppose that $\alpha_k$ is irrational for some $k \in \calK$ satisfying $p\nmid k$ and 
furthermore $p^v k\not \in \calK$ for any $v\in \Z^+$. Then the sequence 
$(f(x))_{x \in \F_q[t]}$ is equidistributed in $\T$.
\end{conj}

In this paper, we make some progress towards this conjecture. Given a set of positive 
integers $\calK$, we define the \textit{shadow} of $\calK$ to be the set  
\[
\mathcal{S}(\calK)=\biggl\{ j\in \Z^+: \text{$p\nmid  \binom{r}{j}$ for some $r\in 
\calK$}\biggr\}.
\]
Here, as usual, we adopt the convention that $\binom{r}{j}=0$ when $j>r$. Note in 
particular that whenever $\calK$ is a set of positive integers, then 
$\calK\subseteq \calS (\calK)$. We provide a convenient interpretation of the shadow of 
$\calK$ in the preamble to Lemma \ref{lem:shadow} that makes for easy computation in 
terms of the base $p$ digital expansions of the elements of $\calK$. We may now 
announce our main equidistribution result, which has no restriction on the degree of the 
polynomial $f(u)$ in question. 

\begin{thm}\label{th:main3}
Let $\calK$ be a finite set of positive integers, suppose that $\alpha_r\in \K_\infty$ for 
$r\in \calK\cup \{0\}$, and define
\[
f(u)=\sum_{r \in \calK\cup\{0\}}\alpha_r u^{r}.
\]
Suppose that  $\alpha_k$ is irrational for some $k \in \calK$ satisfying $p\nmid k$ and 
furthermore $p^v k\not \in \mathcal{S}(\calK)$ for any $v\in \Z^+$. Then the sequence 
$(f(x))_{x \in \Fq[t]}$ is equidistributed in $\T$.
\end{thm}

\begin{ex}\label{example-2}
If $k$ is the largest element of a finite set of positive integers $\calK$, and furthermore 
$p\nmid k$ and $\alpha_k$ is irrational, then Theorem \ref{th:main3} shows that the 
sequence $(f(x))_{x\in \F_q[t]}$ is equidistributed in $\T$. More generally, let 
$f(u)=\sum_{r=0}^k \alpha_r u^r\in \K_\infty[u]$, and suppose that $\alpha_r$ is 
irrational for some integer $r$ with $k/p<r\le k$ and $p\nmid r$. Then, as a direct 
consequence of Theorem \ref{th:main3}, the sequence $(f(x))_{x \in \F_q[t]}$ is 
equidistributed in $\T$.  
\end{ex}

\begin{ex}\label{example-new}
Consider the situation in which $q=3$ and $\calK =\{ 7, 11, 45\}$. A modest computation 
confirms that $\calS(\calK)=\{1,2,3,4,6,7,9,10,11,18,27,36,45\}$. By applying Theorem 
\ref{th:main3}, we see that the sequence 
$(\alpha x^{45}+\beta x^{11}+\gamma x^7)_{x \in \F_q[t]}$ is equidistributed in $\T$ 
if either $\beta$ or $\gamma$ is irrational. This fact does not follow from Example 
\ref{example-2} because $11<45/3$.
\end{ex}

\begin{ex}\label{example-3} 
Suppose that $p>3$ and $\alpha,\beta,\gamma\in \K_\infty$ with $\beta $ irrational. We 
consider the situation with $\calK=\{1,3,3p+1\}$. Since $3p\in \calS(\calK)$, we find that 
Theorem \ref{th:main3} does not imply directly the equidistribution of the sequence 
$(\alpha x+\beta x^3+\gamma x^{3p+1})_{x \in \F_q[t]}$. However, we will prove a 
more general form of Theorem \ref{th:main3} (see Proposition \ref{prop:varmain3} 
below), and from this one can conclude that the above sequence is equidistributed in 
$\T$. In contrast, we are not able to confirm that the sequence 
$(\beta x^3+\gamma x^{4p})_{x \in \F_q[t]}$ is equidistributed in $\T$, although 
Conjecture \ref {conj} suggests that such should be the case.
\end{ex}

\begin{remark}\label{rem:bg1} A result similar to that in Example \ref{example-2} was 
proved independently by Bergelson and Leibman \cite[Corollary 0.5]{bl} using a different 
method. See the discussion concluding this section for a comparison of the latter results 
with those contained in this paper.
\end{remark}

As experts will anticipate, our proof of Theorem \ref{th:main3} is based on an estimate 
for the sum $|\sum _{x \in\G_N}e(f(x))|$ of minor arc type. By combining the large sieve 
inequality  with a generalization of Vinogradov's mean value theorem to the setting of 
$\F_q[t]$, we obtain a Weyl-type estimate which avoids the problematic use of Weyl 
differencing. This approach allows us to surmount the barriers that previously obstructed 
viable conclusions when the degree of $f(u)$ exceeds or is equal to $p$. The assumption 
$p^v k\not \in \mathcal{S}(\calK)$ in Theorem \ref{th:main3} comes from the use of the 
Weyl shift in our minor arc estimate. The latter produces terms whose degrees may lie 
throughout the set $\mathcal{S}(\calK)$, instead of being restricted to the potentially 
smaller set $\calK$ (see equation \eqref{f-shift}). Therefore, we need to consider a mean 
value estimate whose associated indices are elements of $\mathcal{S}(\calK)$. Such an 
``extension of indices'' is a common theme in the study of Diophantine problems. It 
occurs, for example, in Vinogradov's approach to the asympotic formula in Waring's 
problem, where one relates an equation involving $k$-th powers to Vinogradov's system 
of equations having degrees ranging from $1$ to $k$ (see \cite[Section 5.3]{vaughan} for 
more details). In the situation of Theorem \ref{th:main3}, it requires the stronger 
assumption $p^v k\not \in \mathcal{S}(\calK)$, instead of $p^vk \not \in\calK$. Although, 
for this reason, we are unable to prove Conjecture \ref{conj} in general, we can confirm it 
in the special case when $q=p$ (see Corollary \ref{cor:q=pw}). This follows from a more 
general form of Theorem \ref{th:main3} which we present in Proposition 
\ref{prop:varmain3} and Corollary \ref{cor:q=p}.\par

Our equidistribution result is applicable in virtually any situation involving some notion of 
equidistribution for polynomials in $\T$. In particular, in Sections \ref{sec:vdc} and 
\ref{sec:glasner}, we investigate some special sets in $\Fq[t]$ closely related to 
equidistribution and presently less well understood than their integer counterparts. These 
are van der Corput, intersective and Glasner sets. An accessible consequence of this work 
is the following result, which is a consequence of our Theorem \ref{th:vdc1}, established in 
Section 6. 

\begin{thm}\label{th:sarkozy}
Let $\calK$ be a finite set of positive integers, suppose that $a_r\in \F_q[t]$ for 
$r\in \calK\cup \{0\}$, and define
\[
\Phi(u)=\sum_{r \in \calK\cup \{0\}}a_r u^{r}.
\]
Suppose that $\Phi(u)$ has a root modulo $g$ for any $g\in \F_q[t]\setminus \{0\}$. 
Suppose further that $a_k\neq 0$ for some $k\in \calK$ satisfying $p\nmid k$ and 
$p^v k\not\in \mathcal{S}(\calK)$ for any $v \in \Z^+$. Then for any subset 
$\mathcal{A}$ of positive upper density in $\F_q[t]$, there exist distinct elements 
$a$ and $a'$ of $\mathcal{A}$, and some $x\in \F_q[t]$, for which $a-a'=\Phi(x)$. 
\end{thm}

Polynomials $\Phi$ having a root modulo $g$ for any $g \in \F_q[t]\setminus \{0\}$ are 
called \textit{intersective}. The above theorem is an $\F_q[t]$-analog of a result of 
S\'{a}rk\"{o}zy \cite{sarkozy1}. Previously, such a result with no restriction on the degree 
of $\Phi$ was not available, except in cases where $\Phi(0)=0$ \cite{blm} (see also 
\cite{green}). We refer the reader to Section \ref{sec:vdc} for an introduction to 
intersective and van der Corput sets and for the statement of our results.

\begin{remark}\label{rem:bg2} A result similar to Theorem \ref{th:sarkozy} was proved 
independently by Bergelson and Leibman \cite[Theorem 9.5]{bl} using different methods. 
Bergelson and Leibman also addressed a notion of intersective polynomials, although their 
notion differs from ours. It is a nontrivial problem to determine if these two notions are 
one and the same. We refer the reader to Question \ref{q:intersective} in Section 6 and 
the associated discussion for an account of similarities and differences between our 
Theorem \ref{th:sarkozy} and \cite[Theorem 9.5]{bl}. 
\end{remark}

Our next application concerns Glasner sets in $\F_q[t]$. Generalizing a result of Glasner, 
it was shown by Alon and Peres \cite{ap} that given a non-constant polynomial 
$\Phi(u)\in \Z[u]$, for any infinite subset $Y$ of $\R/\Z$ and any $\epsilon>0$, there 
exists $n\in \Z$ such that the set $\Phi(n)Y = \{ \Phi(n)y: y\in Y\}$ intersects any interval 
of length $\epsilon$ in $\R/\Z$. In view of Example \ref{ex:1} and the discussion 
preceding Conjecture \ref{conj}, it is not surprising that an exact analog of the result of 
Alon and Peres over $\F_q[t]$ is \textit{not} true in general. We establish the following 
$\F_q[t]$-analog of the latter result.

\begin{thm}\label{th:glasner}
Let $\calK$ be a finite set of positive integers, suppose that $a_r\in \F_q[t]$ for 
$r\in \calK\cup \{0\}$, and define
\[
\Phi(u)=\sum_{r \in \calK\cup \{0\}}a_r u^{r}.
\]
Suppose that  $a_k\neq 0$ for some $k \in \calK$ satisfying $k>1$ with $p\nmid k$, and 
furthermore $p^v k\not \in \mathcal{S}(\calK)$ for any $v\in \Z^+$. Then for any infinite 
subset $Y\subset \T$ and any $M\in \Z^+$, there exists $x\in \F_q[t]$ having the 
property that the set $\Phi(x)Y$ intersects any cylinder set of radius $q^{-M}$ in $\T$.
\end{thm}

This theorem is a restatement in different language of Theorem \ref{theoremw21}, which is 
itself an immediate consequence of Theorem \ref{th:glasner2}. We refer the reader to 
Section \ref{sec:glasner} for an introduction to Glasner sets and for the statement and 
proof of our results.\par

We conclude this section with a brief comparison between the results of Bergelson and 
Leibman and the results of this paper. As mentioned earlier in Remarks \ref{rem:bg1} and 
\ref{rem:bg2}, some results in this paper were obtained independently by Bergelson and 
Leibman \cite{bl}, at about the same time as an earlier version of this 
memoir\footnote{The first version of our paper was posted on arxiv 
(https://arxiv.org/abs/1311.0892) in November 2013.}, using rather different methods. 
The approach of Bergleson and Leibman is qualitative and very general. Their main result, 
\cite[Theorem 0.3]{bl}, concerns multi-dimensional tori $\T^c$. It asserts that any 
(multi-variate) polynomial sequence in $\T^c$ is equidistributed in a finite union of cosets 
of a subgroup of $\T^c$. It also gives a condition for when a polynomial sequence is 
equidistributed in the full torus. However, this condition is not easy to check in practice for 
a given polynomial and we do not know if \cite[Theorem 0.3]{bl} implies our Theorem 
\ref{th:main3}. There are two important features of our own work. First, our method 
(which relies on the large sieve inequality and Vinogradov's Mean Value Theorem) offers 
scope for \textit{quantitative} applications. For example, it was used by Yamagishi in work 
on Diophantine approximation \cite{yamagishi2} and Waring's problem over $\F_q[t]$ 
\cite{yamagishi1}. Second, the flexibility of our approach makes it applicable to variants of 
Weyl sums in which summands are restricted in various ways. Indeed, in recent work with 
Zhenchao Ge \cite{gll}, the first and second authors extend the methods of the current 
paper to study Weyl sums over the set ${\mathbb I}_q$ of monic irreducible elements in 
$\Fq[t]$, thereby obtaining equidistribution results for the sequence 
$(f(x))_{x \in \mathbb I_q}$ with concomitant conclusions for allied Diophantine and 
combinatorial problems.\par 

This paper is organized as follows. In Section \ref{sec:prelim} we introduce the 
preliminary infrastructure needed to prove our results. We prove an estimate of minor arc 
type in Section \ref{sec:weyl} and derive an extension of this conclusion suitable for our 
subsequent applications in Section \ref{sec:weyl2}. Then, in Section 
\ref{sec:equidistribution}, we apply these estimates to prove Theorem \ref{th:main3}. 
Finally, in Sections \ref{sec:vdc} and \ref{sec:glasner}, we discuss applications of our 
equidistribution results to van der Corput, intersective and Glasner sets over $\F_q[t]$.\par

\noindent{\bf Acknowledgements:} We are grateful to Vitaly Bergelson for explaining 
aspects of the paper \cite{bl}, and to Bhawesh Mishra for interesting conversations related 
to the topic of our paper and directing us to \cite{AB2023}.

\section{Preliminaries}\label{sec:prelim}  
We begin this section by reviewing an orthogonality relation for the function $e(\cdot)$ 
defined in Section 1. As is explained in \cite[Lemma 7]{Ku}, for example, when 
$\alpha\in \K_{\infty}$, we have
\begin{equation}\label{eq:orthogonal1}
\sum_{x \in \G_N}e(x\alpha)=\begin{cases}
q^N,&\text{when $\ord \{\alpha \}<-N$,}\\
0,&\text{when $\ord \{\alpha \}\ge -N$.}
\end{cases}
\end{equation}              
Therefore, for any polynomials $a, g\in \F_q[t]$ with $g\neq 0$, we have
\begin{equation}\label{eq:orthogonal2}
\sum_{\ord x<\ord g}e\left(\frac{xa}{g}\right)=
\begin{cases}
|g|,&\text{when $a\equiv 0\mmod{g}$,}\\
0,&\text{otherwise.}
\end{cases}
\end{equation}

As promised in the preamble to the statement of Theorem \ref{th:main3}, we now 
interpret the shadow $\calS(\calK)$ of a set of indices $\calK$ in a manner that eases 
explicit computations. First, given $j,r\in \Z^+$, we write $j\preceq_p r$ when 
$p\nmid \binom{r}{j}$. By Lucas' theorem, the latter holds precisely when all of the digits 
of $j$ in base $p$ are less than or equal to the corresponding digits of $r$. From this 
characterization, it is easy to see that the relation $\preceq_p$ defines a partial order on 
$\Z^+$. Note in particular that if $j\preceq_p r$, then we necessarily have $j\leq r$. 
Equipped with this notation, we see that
\begin{equation}\label{w1}
\mathcal{S}(\calK)=\left\{ j\in \Z^+: \text{$j\preceq_p r$ for some $r\in \calK$}\right\}.
\end{equation}
This interpretation makes clear the origin of the elements of $\calS(\calK)$ occurring in 
Example \ref{example-new}. Thus, in transparent notation, the base $10$ number $7$ 
has base $3$ expansion $(21)_3$, and thus $\calS(\calK)$ must contain the numbers 
$7=(21)_3$, $6=(20)_3$, $4=(11)_3$, $3=(10)_3$ and $1=(1)_3$. Likewise, the base 
$10$ number $11$ has base $3$ expansion $(102)_3$, and hence $\calS(\calK)$ must 
contain the numbers $11=(102)_3$, $10=(101)_3$, $9=(100)_3$, $2=(2)_3$ and 
$1=(1)_3$. Finally, the base $10$ number $45$ has base $3$ expansion $(1200)_3$, and 
hence $\calS(\calK)$ contains the numbers $45=(1200)_3$, $36=(1100)_3$, 
$27=(1000)_3$, $18=(200)_3$ and $9=(100)_3$.\par

Our conclusions concerning estimates of Weyl-type and associated equidistribution results 
extend beyond those announced in Theorem \ref{th:main3}. For ease of reference, we 
take the opportunity here to collect together the definitions of certain subsets of the set of 
indices $\calK$ making an appearance later in this paper. First, define 
\begin{equation}\label{eq:k-star}
\calK^*=\left\{ k \in \calK: \text{$p\nmid k$ and $p^v k\not \in \mathcal{S}(\calK)$ for 
any $v\in \Z^+$}\right\}.
\end{equation}
The set $\calK^*$ is therefore the subset of $\calK$ that is compatible with an application 
of Theorem \ref{th:main3}, namely the subset of $\calK$ consisting of indices, no 
non-trivial $p$-power multiple of which lies in the shadow of $\calK$. The set 
$\calK\setminus \calK^*$ consists of indices not immediately accessible to Theorem 
\ref{th:main3}. However, if we throw out the accessible exponents $\calK^*$ and treat 
the remaining set $\calK\setminus \calK^*$ in isolation, it may well be that a new set 
$(\calK\setminus \calK^*)^*$ can be identified itself accessible to Theorem 
\ref{th:main3}, and this process can be iterated. We are therefore led to define the set 
$\tK$ as follows. We put $\calK_0 = \calK$, and inductively define for each $n \ge 1$ the 
set
\[
\calK_n=\calK_{n-1}\setminus \calK_{n-1}^*.
\] 
We then define the set of indices 
\begin{equation}\label{eq:ktilde}
\tK =\bigcup_{n=0}^{\infty} \calK_{n}^*.
\end{equation}
We show in Proposition \ref{prop:varmain3} that the conclusion of Theorem \ref{th:main3} 
may be extended so that indices $k$ remain accessible throughout the set $\tK$, instead of 
being constrained to lie in $\calK^*$.\par
 
Next, consider a set $\calK \subset \Z^+$. We say that an element $k\in \calK$ is 
{\it maximal} if it is maximal with respect to the partial ordering $\preceq_p$. Thus, for 
any $r\in \calK$, one has either $r\preceq_p k$ or else $r$ and $k$ are not 
comparable. We record for future reference the following observations concerning the 
partial ordering $\preceq_p$.

\begin{lem}\label{lem:shadow}
Suppose that $\calK \subset \Z^{+}$. Then the following hold.
\begin{enumerate}[topsep=-5pt,itemsep=-1ex,partopsep=1ex,parsep=1ex]
\item[(a)] The index $k$ is maximal in $\mathcal{S}(\calK)$ whenever $k$ is maximal in 
$\calK$;
\item[(b)] One has $\calK^* \subset \mathcal{S}(\calK)^*$;
\item[(c)] If $k\in \calK^*$, and $j\in \calK$ satisfies $ k\preceq_pj$, then $j\in \calK^*$.
\end{enumerate}
\end{lem}

\begin{proof} The maximality property (a) is immediate from the definition of 
$\mathcal{S}(\calK)$. Property (b), meanwhile, follows from the definition 
\eqref{eq:k-star} of $\calK^*$ on observing that 
$\mathcal{S}(\mathcal{S}(\calK))=\mathcal{S}(\calK)$. Finally, under the hypotheses of 
part (c), we have $p\nmid k$ and $p\nmid \binom{j}{k}$. By Lucas' theorem, it follows 
that $p\nmid j$. A second application of Lucas' theorem reveals that for any $v\in \Z^+$, 
we have $p^v k\preceq_p p^v j$. If we were to have $p^v j\in \mathcal{S}(\calK)$ for 
some $v\in \Z^+$, then for some $r\in \calK$ we would have 
$p^v k\preceq_p p^v j\preceq_p r$, whence $k\not \in \calK^*$, yielding a contradiction. 
So $p^v j\not \in \mathcal{S}(\calK)$ for any $v\in \Z^+$, and we conclude that 
$j\in \calK^*$.
\end{proof}

In order to state the version of the large sieve inequality that we employ to derive a 
minor arc estimate, we must introduce some notation. Suppose that 
$\Gamma \subset \K_\infty$. We say that the elements of $\Gamma$ are 
{\it $q^\delta$-spaced in $\T$} if, for any distinct elements 
$\gamma_1,\gamma_2\in \Gamma$, we have 
$\ord \{ \gamma_1-\gamma_2\}\ge \delta$.

\begin{thm}\label{largesieve}
Let $K$ and $N$ be positive integers. Suppose that $\Gamma\subset \K_\infty$ is a 
$q^{-K}$-spaced set in $\T$. Consider a sequence $(b_x)_{x \in \F_q[t]}$ of complex 
numbers, and when $\beta \in \K_\infty$ define 
\[
\mathcal{S}(\beta)=\sum_{x \in \G_N}b_x\, e(x\beta).
\]
Then
\[
\sum_{\gamma \in \Gamma}|\mathcal{S}(\gamma)|^2\le 
\max\left\{ q^N,q^{K-1}\right\} \sum_{x\in \G_N} |b_x|^2.
\]
\end{thm}

\begin{proof} This is Hsu \cite[Theorem 2.4]{hsu}.
\end{proof}

In order to apply Theorem \ref{largesieve}, we employ a construction from 
\cite{lw-waring}. It is convenient in this setting to introduce some further notation.

\begin{defn}
Suppose that $k\in \mathbb{Z}^+$ and $g\in \F_q[t]\setminus\{0\}$. We say that a set 
of monic polynomials $\mathcal{L}\subset \F_q[t]$ is a \textit{$(k,g)$-set} if, for any 
$\ell_1,\ell_2\in \mathcal{L}$, one has $\ell^k_1\equiv \ell^k_2\mmod{g}$ if and only if 
$\ell_1\equiv \ell_2\mmod{g}$.
\end{defn}

The next lemma allows us to partition a given finite subset of $\F_q[t]$ into a small 
number of $(k,g)$-sets.

\begin{lem}\label{hensel}
Let $k$ be a positive integer satisfying $p\nmid k$. Also, let $g\in \F_q[t]$, and suppose 
that $A$ is a subset of $\F_q[t]$, all of whose elements are coprime to $g$. Then for 
each $\epsilon>0$, the set $A$ can be partitioned into $O_{k,q,\epsilon}(|g|^\epsilon)$ 
subsets, each of which is a $(k,g)$-set.
\end{lem}

\begin{proof} This is essentially \cite[equation (12.4)]{lw-waring}, though for completeness 
we include a proof. We begin with an estimate for the number of solutions of a certain 
polynomial congruence. Working under the hypotheses of the statement of the lemma, 
when $a\in \mathbb{F}_{q}[t]$, denote by $J(g,a)$ the number of solutions of the 
congruence $x^k\equiv a\mmod{g}$ with $\deg(x)<\deg(g)$ and $(x,g)=1$. Thus, 
necessarily, one has $(a,g)=1$. Then we claim that $J(g,a)\leq k^{\omega(g)}$, where 
$\omega(g)$ denotes the number of distinct monic irreducible factors of $g$. For each 
$a\in \Fq[t]$, we write $\{ x_1(a),\ldots ,x_J(a)\}$ for the set of solutions of the above 
congruence, where $J=J(g,a)$ and the elements $x_i(a)$ are distinct for $1\le i\le J$. Then 
$A$ can be partitioned into the sets
\[
A_i=\{ x\in A: \text{there exists $a\in \Fq[t]$ such that $J(g,a)\ge i$ and 
$x\equiv x_i(a)\mmod{g}$}\},
\]
for $1\le i\le k^{\omega(g)}$, each of which is a $(k,g)$-set. The conclusion of the lemma 
follows by means of the familiar estimate
\[
\omega(g) \leq \log_2 d(g) \ll_q \frac{\deg g}{\log \deg g},
\] 
where $d(g)$ denotes the number of divisors of $g$ (see for example 
\cite[Lemma 5]{Le2011}).\par

We now set about confirming the above claim. For each irreducible polynomial $\ell$ with 
$\ell\mid g$, the congruence $x^k\equiv a\mmod{\ell}$ has at most $k$ solutions. Thus, 
since $p\nmid k$, it follows from Hensel's lemma that for any $r\geq 2$, each solution of 
$x^k\equiv a\mmod{\ell}$ lifts uniquely to a corresponding solution modulo $\ell^r$. 
Factoring $g$ as a product of powers of irreducible polynomials in the form 
$\prod \ell_j^{r_j}$, and counting solutions modulo $\ell_j^{r_j}$ for each $j$, we deduce 
via the Chinese Remainder Theorem that there are at most $k^{\omega(g)}$ solutions 
modulo $g$. This completes the proof of the lemma.
\end{proof}

We next state a mean value theorem for a system of equations having indices defined by 
the elements of the set $\mathcal{S}(\calK)$ defined in \eqref{w1}. For $N\in \Z^+$, 
denote by $J_s(\mathcal{S}(\calK);N)$ the number of solutions of the system
\[
u_1^j+\cdots +u_s^j=v_1^j+\cdots +v_s^j\quad (j\in \mathcal{S}(\calK)),
\]
with $u_r, v_r\in \GN$ $(1\le r\le s)$. Since 
$(u_1+\cdots +u_s)^p=u_1^p+\cdots +u_s^p$, these equations are not always 
independent. To obtain independence, we consider the set
\begin{equation}\label{S(K)-prime}
\mathcal{S}(\calK)'=\left\{i \in \Z^+:\text{$p\nmid i$ and $p^v i\in \mathcal{S}(\calK)$ 
for some $v \in \Z^+\cup\{0\}$}\right\}.
\end{equation}
We note that when $j=p^v i$ with $p\nmid i$, we have 
$u_1^j+\cdots +u_s^j=(u_1^i+\cdots +u_s^i)^{p^v}$. It therefore follows that 
$J_s(\mathcal{S}(\calK);N)$ also counts the number of solutions of the system
\[
u_1^i+\cdots +u_s^i=v_1^i+\cdots +v_s^i\quad (i\in \mathcal{S}(\calK)'),
\]
with $u_r, v_r\in \GN$ $(1\le r\le s)$. We shall find it useful to define three quantities 
associated with this system of equations, namely
\begin{equation}\label{w2}
\psi(\calK)=\text{card}\,\mathcal{S}(\calK)',\quad 
\phi(\calK)=\max_{i \in  \mathcal{S}(\calK)'}\, i \quad \text{and}\quad 
\kappa(\calK)=\sum_{i \in \mathcal{S}(\calK)'}i.
\end{equation}
Where the intended meaning is unambiguous, we drop mention of $\calK$ from this 
notation without comment. The following result gives an upper bound on 
$J_s(\mathcal{S}(\calK);N)$. 

\begin{thm}\label{vmt}
Suppose that $s\ge \psi(\phi+1)$. Then for any $\epsilon>0$, there exists a constant 
$C_1=C_1(s;\calK;\epsilon;q)>0$ such that 
\[
J_s(\mathcal{S}(\calK);N)\le C_1(q^N)^{2s-\kappa +\epsilon }.
\]
\end{thm}

\begin{proof} Observe that whenever $j\in \mathcal{S}(\calK)$, and $i\in \Z^+$ satisfies  
$i\preceq_pj$, one has $i\in \mathcal{S}(\calK)$. Therefore, the set $\mathcal{S}(\calK)$ 
satisfies the inclusion relation defined in Condition$\star$ of \cite[Section 1]{klz}. The 
desired conclusion therefore follows as a special case of \cite[Theorem 1.1]{klz}. 
\end{proof}

We remark that a multidimensional generalization of Theorem \ref{vmt} can be found in 
\cite{klz}. Meanwhile, the condition $s\ge \psi(\phi+1)$ of this theorem can be refined, as is 
shown in \cite{lw}.\par

We now recall some facts about continued fractions in $\K_\infty$ needed in our proof of 
Theorem \ref{th:main3}. For any irrational element $\alpha$ lying in $\K_{\infty}$, we can 
write $\alpha$ as an infinite continued fraction in the form
\[
\alpha =b_0+\frac{1}{b_1+\frac{1}{b_2+\cdots}}=[b_0;b_1,b_2,\ldots],
\]
with $b_i\in \Fq[t]$ and $\ord b_i>0$ $(i\ge 1)$. When $\alpha$ is a rational element of 
$\K_\infty$, meanwhile, one may write $\alpha$ as a finite continued fraction of the form
\[
\alpha =b_0+\frac{1}{b_1+\frac{1}{b_2+\frac{1}{\cdots +\frac{1}{b_n}}}}=
[b_0;b_1,b_2,\ldots,b_n],
\]
with $b_i\in \Fq[t]$ and $\ord b_i>0$ $(1\le i\le n)$. We note that continued fraction 
expansions in $\K_{\infty}$ are uniquely defined. We define two sequences 
$(a_n)_{n\geq -2}$ and $(g_n)_{n\geq -2}$ in $\Fq[t]$ recursively by putting
\[
a_{-2}=0,\quad g_{-2}=1,\quad a_{-1}=1,\quad g_{-1}=0,
\]
and for all $n \ge 0$,
\[
a_n=b_n a_{n-1}+a_{n-2}\quad \textup{and} \quad g_n =b_n g_{n-1}+g_{n-2}.
\]
Then for all $n\geq 0$, we have
\[
g_na_{n-1}-a_{n}g_{n-1}=(-1)^n\qquad \textup{and}\qquad 
[b_0;b_1,\ldots ,b_n]=a_n/g_n.
\]
The fractions $a_n/g_n$ $(n \ge 0)$ are called the \textit{convergents} of $\alpha$. An 
inductive argument shows that the sequence $(\ord g_n)_{n\ge 0}$ is strictly increasing. 

\begin{prop}\label{prop:cont}
Suppose that $\alpha\in \K_\infty$. Then the convergents $a_n/g_n$ $(n \ge 0)$ of 
$\alpha$ satisfy the following properties.
\begin{enumerate}[topsep=-5pt,itemsep=-1ex,partopsep=1ex,parsep=1ex]
\item[(a)] One has $\ord (g_n\alpha -a_n)=-\ord g_{n+1}$ $(n \ge 0)$.
\item[(b)] If $a,g\in \Fq[t]$ satisfy $\ord (g\alpha -a)<-\ord g$, then $a/g$ is a convergent 
of $\alpha$.
\end{enumerate}
\end{prop}

\begin{proof} See \cite[Section 1]{schmidt}.
\end{proof}

The conclusion (b) of Proposition \ref{prop:cont} is sometimes referred to as Legendre's 
theorem. The following lemma concerns elements of $\K_\infty$ well-approximated by 
rationals.

\begin{lem}\label{lem:diophantine}
Let $\alpha \in \K_{\infty}$. Suppose that there exists a constant $\kappa >1$ such that, 
for all sufficiently large $N$, there exist $a\in \Fq[t]$ and $g\in \Fq[t]\setminus \{0\}$ with 
$\ord (g\alpha-a)\le -\kappa N$ 
and $\ord g<N$. Then $\alpha$ is rational.
\end{lem}

\begin{proof}
Suppose that $\alpha$ is irrational and $a_n/g_n$ $(n \ge 0)$ are the convergents of 
$\alpha$. Since $\alpha$ is irrational, we have $\lim_{n \rightarrow \infty}\ord g_n=\infty$. 
We take $n$ sufficiently large and put $N=\ord{g_n}$. By hypothesis, there exist 
$a\in \Fq[t]$ and $g\in \Fq[t]\setminus\{0\}$ such that $\ord g<N$ and
\begin{equation}\label{w3}
\ord (g\alpha -a)\le -\kappa N<-\ord g_n=-N<-\ord g.
\end{equation}
It therefore follows from Proposition \ref{prop:cont}(b) 
that $a/g$ is a convergent of $\alpha$. But 
$\ord g<N=\ord g_n$ and the sequence $(\ord g_n)_{n\ge 0}$ is strictly increasing, so 
there exists $m\in \Z^+\cup\{0\}$ with $m<n$ such that $a=a_m$ and $g=g_m$. 
However, we find from Proposition \ref{prop:cont}(a) that 
\[
\ord (g\alpha -a)=\ord(g_m\alpha -a_m)=-\ord (g_{m+1})\ge -\ord g_n,
\]
and this contradicts \eqref{w3}. We thus conclude that $\alpha$ is rational. 
\end{proof}

We end this section by recalling Weyl's criterion for equidistribution  in $\Fq[t]$.

\begin{thm}\label{th:weyl}
The sequence $(a_x)_{x \in \Fq[t]}\subset \K_\infty$ is equidistributed in $\T$ if and only if 
for any $m\in \Fq[t]\setminus \{0\}$, we have
\[
\lim_{N \rightarrow \infty}\frac{1}{q^N}\Biggl| \sum_{x\in \GN}e(ma_x)\Biggr|=0.
\]
\end{thm}

\begin{proof} This is Carlitz \cite[Theorem 4]{carlitz}.
\end{proof}

\section{A Weyl-type estimate}\label{sec:weyl}
Our goal in this section is the proof of an estimate of minor arc type for a certain 
exponential sum. In advance of the statement of this estimate, we recall the definition 
\eqref{eq:k-star} of the set $\calK^*$.  

\begin{thm} \label{th:main1}
Fix $q$ and a finite set $\calK \subset \Z^+$. There exist positive constants $c$ and $C$, 
depending only on $\calK$ and $q$, such that the following holds. Let $\epsilon>0$ and let 
$N$ be sufficiently large in terms of $\calK$, $\epsilon$ and $q$. Suppose that 
$f(u)=\sum_{r\in \calK\cup\{0\}}\alpha_r u^{r}$ is a polynomial with coefficients in 
$\K_\infty$ satisfying the bound
\[
\biggl|\sum_{x \in \GN}e(f(x))\biggr| \geq q^{N -\eta },
\]
for some positive number $\eta$ with $\eta \le cN$. Then, for each maximal $k\in \calK^*$, 
there exist $a\in \Fq[t]$ and monic $g\in \Fq[t]$ having the property that
\[
\ord (g\alpha_k-a)<-kN+\epsilon N+C\eta \quad \textrm{and}\quad \ord g \leq \epsilon 
N+C\eta.
\] 
\end{thm}

We remark that an $\epsilon$-free version of this conclusion could be derived by making 
use of major arc approximations to the exponential sum under consideration. We direct the 
interested reader to \cite[Lemma 2.1]{Woo2003} for a model of the kind of argument that 
would be required to achieve such a conclusion. Observe also that in Theorem 
\ref{th:main1}, the coefficient $\alpha_k$ plays the role of the leading coefficient of the 
polynomial, and might be regarded as the ``true'' $\Fq[t]$-analog of the leading coefficient. 
Furthermore, clearly, if $k$ is the greatest element in $\calK$, then $k$ is maximal in 
$\calK$. However, a set may have more than one maximal element. For example, if $p=2$ 
and $\calK = \{1,3,5,9\}$ then $9$, $5$, and $3$ are all maximal elements of $\calK$ 
and they all satisfy the hypothesis of Theorem \ref{th:main1}.\par

We require two auxiliary lemmas in our proof of Theorem \ref{th:main1}. First, we recall 
a familiar lemma employing Weyl shifts of a form suitable for our subsequent deliberations.  

\begin{lem}\label{weylshift} 
Let $\mathcal{A}$ be a multiset of elements from $\GN$, and write $|\mathcal{A}|$ for 
$\text{card}(\mathcal{A})$. Then we have
\[
\usum_{x\in \GN}e(f(x))= |\mathcal{A}|^{-1}\usum_{x\in \GN}\usum_{y \in \mathcal{A}}
e(f(y-x)).
\]
\end{lem}

\begin{proof} For $y \in \GN$, it follows via a change of variable that
\[
\sum_{x\in \GN}e(f(x))=\sum_{x \in \GN}e(f(y-x)).
\]
Thus, it follows that 
\[
|\mathcal{A}|\sum_{x\in \GN}e(f(x))=\usum_{y \in \mathcal{A}}
\sum_{x \in \GN}e(f(y-x))=\sum_{x \in \GN}\usum_{y\in \mathcal{A}}e(f(y-x)),
\]
and the desired conclusion is immediate.
\end{proof}

Consider a finite subset $\calK$ of $\Z^{+}$ and its shadow $\mathcal{S}(\calK)$. Let 
$f(u)=\sum_{r\in \calK\cup \{0\}}\alpha_r u^{r}$ be a polynomial with coefficients in 
$\K_\infty$, and write ${\boldsymbol \alpha}$ for $\{ \alpha_r\}_{r\in \calK}$. For any 
$r\in \calK$, we have
\[
(y-x)^r=\usum_{j\preceq_p r}\binom{r}{j} y^j(-x)^{r-j}+(-x)^r.
\]
Therefore, if $k$ is maximal in $\calK$, then for a fixed $x\in \GN$ there exist
\[
\gamma_0=\gamma_0(\alpha_0,{\boldsymbol \alpha};x)\in \mathbb{K}_\infty\qquad 
\text{and}\qquad \gamma_j=\gamma_j({\boldsymbol \alpha};x)\in \mathbb{K}_\infty
\quad (j\in \mathcal{S}(\calK)\setminus \{k\})
\]
such that 
\begin{equation}\label{f-shift}
f(y-x)=\alpha_k (y-x)^k+\usum_{r\in \calK\setminus \{k\}}\alpha_r (y-x)^r+\alpha_0
=\alpha_k y^k+\usum_{j\in \mathcal{S}(\calK)\setminus \{k\}}\gamma_j y^j+\gamma_0.
\end{equation}
The next lemma provides a conclusion occurring within the argument of the proof of 
\cite[Lemma 12.1]{lw-waring}.

\begin{lem}\label{spacing}
Let $M \in \Z^+$ with $M \le N$, and let $k \in \Z^+$ with $p\nmid k$ and 
$\alpha _k\in \mathbb{K}_\infty$. Suppose that $a, g\in \F_q[t]$ with $(a,g)=1$ and 
$\ord(g\alpha_k-a)<-kM$, and suppose further that either $\ord(g\alpha_k-a)\ge M-kN$ or 
$\ord g>M$. Finally, let $\mathcal{L}_0$ be a $(k, g)$-subset of monic polynomials of 
degree $M$. Then the points $\{\alpha_k l^k:l\in \mathcal{L}_0\}$ are spaced at least 
$\min\{|g|^{-1},q^{k(M-N)}\}$ apart in $\mathbb{T}$.  
\end{lem}

\begin{proof} Suppose that $l_1,l_2\in \mathcal{L}_0$ with $l_1\not \equiv l_2\mmod{g}$. 
Then, since $\mathcal{L}_0$ is a $(k,g)$-subset, we have 
$l_1^k\not \equiv l_2^k\mmod{g}$. Write $\alpha_k=a/g+\beta$. Then
\[
\ord \{ \alpha_k(l_1^k-l_2^k)\} =\ord \{ a(l_1^k-l_2^k)/g+\beta (l_1^k-l_2^k)\}.
\]
Since $\ord (g\beta)<-kM$ and $\ord l_1=\ord l_2=M$, we have
\[
\ord \{ \beta(l_1^k-l_2^k)\}<-kM-\ord g+kM=-\ord g.
\]
Also, since $l_1^k\not \equiv l_2^k\mmod{g}$ and $(a,g)=1$, we have 
\[
\ord \{a(l_1^k-l_2^k)/g\} \ge -\ord g.
\]
We therefore deduce that 
\begin{equation}\label{spacing1}
\ord \{ \alpha_k(l_1^k-l_2^k)\}=\ord \{ a(l_1^k-l_2^k)/g\}\ge -\ord g.
\end{equation}

\par We now divide into cases, according to the size of $\ord g$. 

\noindent \textbf{Case 1.} Suppose first that $\ord g>M$. In this case, the elements of 
$\mathcal{L}_0$ are distinct $\mmod{g}$. Consequently, by \eqref{spacing1}, the points 
$\alpha_k l^k$ are spaced at least $|g|^{-1}$ apart in $\mathbb{T}$.

\noindent \textbf{Case 2.} If instead $\ord g\le M$, then the hypotheses of the lemma 
ensure that one has $\ord (g\alpha_k-a)\ge M-kN$. When $l_1,l_2\in \mathcal{L}_0$ satisfy 
the condition $l_1\not \equiv l_2\mmod{g}$, then it follows from \eqref{spacing1} that 
$\alpha l_1^k$ and $\alpha l_2^k$ are spaced at least $|g|^{-1}$ apart in $\mathbb{T}$. 
Otherwise, when $l_1\equiv l_2\mmod{g}$, the bounds $\ord (g\alpha_k-a)<-kM$ and 
$\ord (g\alpha_k-a)\ge M-kN$ lead to the relation 
\begin{align}
\ord \{\alpha_k(l_1^k-l_2^k)\}&=\ord \{(\alpha_k-a/g)(l_1^k-l_2^k)\}\notag \\
&=\ord \big((\alpha_k-a/g)(l_1^k-l_2^k)\big)\notag \\
&\ge M-kN-\ord g+\ord (l_1^k - l_2^k).\label{spacing2}
\end{align}
We note that
\[
\ord (l_1^k-l_2^k)=\ord (l_1-l_2)+\ord (l_1^{k-1}+l_1^{k-2}l_2+\cdots +l_2^{k-1}).
\]
If $l_1\neq l_2$ and $l_1\equiv l_2\mmod{g}$, we have $\ord (l_1-l_2)\ge \ord g$. 
Furthermore, since the elements of $\mathcal{L}_0$ are monic 
and of degree $M$, the term $l_1^{k-1}+l_1^{k-2}l_2+\cdots +l_2^{k-2}$ is of degree 
$(k-1)M$ with leading coefficient $k$. Since $p\nmid k$, we have
\[
\ord (l_1^{k-1}+l_1^{k-2}l_2+\cdots +l_2^{k-1})=(k-1)M.
\]
On combining the above two estimates, we obtain the lower bound 
\[
\ord (l_1^k-l_2^k)\ge \ord g+(k-1)M,
\]
and hence we infer from (\ref{spacing2}) that
\[
\ord \{\alpha_k(l_1^k-l_2^k)\} \ge k(M-N).
\]
In this case, therefore, we find that $\alpha l_1^k$ and $\alpha l_2^k$ are spaced at least 
$q^{k(M-N)}$ apart in $\mathbb{T}$.\par

Combining the bounds obtained in the two respective cases, we conclude that for any 
distinct elements $l_1,l_2\in\mathcal{L}_0$, the points $\alpha_k l_1^k$ and 
$\alpha_k l_2^k$ are spaced at least $\min\{|g|^{-1}, q^{k(M-N)}\}$ apart in $\T$. This 
completes the proof of the lemma.
\end{proof}

We are now ready to prove Theorem \ref{th:main1}.

\begin{proof}[Proof of Theorem \ref{th:main1}]
We first note that should Theorem \ref{th:main1} hold for the polynomial 
$f(u)-\alpha_0=\sum_{r\in \calK}\alpha_ru^r$, then it holds also for $f(u)$. There is 
consequently no loss of generality in assuming that $\alpha_0=0$. Next, let $k$ be a 
maximal element of $\calK$ satisfying $p\nmid k$ and $p^vk\not \in \mathcal{S}(\calK)$ 
for any $v\in \Z^+$. Let $\alpha_k\in \K_\infty$ and consider $M \in \Z^+$ with $2M\le N$. 
By Dirichlet's approximation theorem in $\F_q[t]$ (see \cite[Lemma 3]{Ku}), there exist 
$a\in \F_q[t]$ and monic $g\in \F_q[t]$ with
\[
(a,g)=1,\quad \ord(g\alpha_k-a) <-kM\quad \text{and}\quad \ord g \le kM.
\]
Suppose that either
\begin{equation}\label{w9}
\ord(g\alpha_k-a)\ge M-kN\quad \text{or}\quad \ord g>M.
\end{equation}
We will show that, for $M$ suitably chosen, such an assumption leads to an upper bound 
for $\big|\sum_{x \in \GN}e(f(x))\big|$, which contradicts the lower bound asserted in the 
statement of the theorem.\par

Let $\mathcal{L}$ be the set of monic irreducible polynomials $l$ satisfying $\ord l=M$ and 
$(l,g)=1$. Since $\ord g\le kM$, the polynomial $g$ has at most $k$ irreducible factors of 
degree $M$. It therefore follows from the prime number theorem in $\F_q[t]$ that when 
$M$ is sufficiently large in terms of $k$ (and thus also $\calK$) and $q$, we have 
\[
q^M/(2M)\le \text{card}(\mathcal{L})\le q^M/M.
\]
Let $\calA$ be the multiset
\begin{equation}\label{w4}
\mathcal{A}=\big\{y\in \G_N:\text{$y=lw$ with $l\in \mathcal{L}$ and 
$w\in \G_{N-M}$}\big\},
\end{equation}
where the multiplicity of each element $y$ of $\mathcal{A}$ is equal to the number of its 
representations $y=lw$. Then
\[
|\mathcal{A}|=\text{card}(\mathcal{A})\geq q^{N-M}\cdot q^M/(2M)=q^N/(2M).
\] 
By Lemma \ref{weylshift} and (\ref{f-shift}), we therefore find that 
\begin{align*}
\biggl| \usum_{x\in \GN}e(f(x))\biggr| &\le 2Mq^{-N}\biggl| \usum_{x \in \GN}
\usum_{y\in \mathcal{A}}e\Big(\alpha_k y^k+
\usum_{j\in \mathcal{S}(\calK)\setminus \{k\}}\gamma_j({\boldsymbol \alpha};x)y^j\Big) 
\biggr| \\
&\le 2M\max_{x \in \GN}\biggl|\usum_{y \in  \mathcal{A}} e\Big(\alpha_k y^k+
\usum_{j \in \mathcal{S}(\calK)\setminus \{k\}}\gamma_j({\boldsymbol \alpha};x)
y^j\Big)\biggr|. 
\end{align*}
For $j\in \calS (\calK)\setminus \{k\}$, fix $\gamma_j=\gamma_j({\boldsymbol \alpha};x)$ 
to be the element of $\K_\infty$ corresponding to the choice of $x$ which maximizes the 
expression on the right hand side here.\par

Recall the definitions \eqref{w2} of $\psi$ and $\phi$, and let $s$ be a positive integer with 
$s\geq \psi\phi+\psi$. Then in view of \eqref{w4}, an application of H\"older's inequality 
delivers the bound
\[
\biggl|\usum_{x\in \GN}e(f(x))\biggr|^{2s}\le (2M)^{2s}(q^M/M)^{2s-1}
\usum_{l\in \mathcal{L}}\,\biggl| \sum_{w\in  \G_{N-M}}e\Big(\alpha_k (lw)^k+
\usum_{j\in \mathcal{S}(\calK)\setminus \{k\}}\gamma_j (lw)^j\Big)\biggr|^{2s}.
\]
Let $\epsilon>0$ be arbitrary. By Lemma \ref{hensel}, there exists a constant 
$C_1=C_1(k,q,\epsilon)>0$ such that the set $\mathcal{L}$ can be divided into 
$L\leq C_1|g|^\epsilon$ subsets $\mathcal{L}_1,\ldots ,\mathcal{L}_L$, having the 
property that $\mathcal{L}_i$ is a $(k,g)$-set for $1\le i\le L$. Then there exists 
$r\in \Z^+$ with $r \le L$ for which 
\begin{equation}\label{w5}
\biggl| \usum_{x\in \GN}e(f(x))\biggr|^{2s}\le 2^{2s}M(q^M)^{2s-1}C_1|g|^\epsilon 
\Psi,
\end{equation}
where
\begin{equation}\label{w6}
\Psi=\usum_{l\in \mathcal{L}_r}\biggl| \usum_{w\in \G_{N-M}}e\Big(\alpha_k (lw)^k+
\usum_{j\in \mathcal{S}(\calK)\setminus \{k\}}\gamma_j(lw)^j\Big) \biggr|^{2s}.
\end{equation}

Let $\mathcal{S}(\calK)'$ be the relation of the shadow set defined in \eqref{S(K)-prime}.  
For ${\bf h}=(h_i)_{i\in \mathcal{S}(\calK)'}$ with $h_i\in \F_q[t]$, let $b(\bf{h})$ denote 
the number of solutions of the system 
\[
w_1^i+\cdots +w_s^i=h_i\quad (i \in \mathcal{S}(\calK)'),
\]
with $w_r\in \G_{N-M}$ $(1\le r\le s)$. For $i\in \mathcal{S}(\calK)'$, we have 
$h_i\in \G_{i(N-M)}$. Furthermore, for $j= p^v i\in \mathcal{S}(\calK)$, with 
$i\in \mathcal{S}(\calK)'$ and $v\in \Z^+$, we have $w_1^j+\cdots +w_s^j=h_i^{p^v}$. 
Therefore, by defining $h_j=h_i^{p^v}$, we see that $b(\bf{h})$ also counts the number 
of solutions of the system
\begin{equation}\label{w7}
w_1^j+\cdots +w_s^j=h_j\quad (j \in \mathcal{S}(\calK)),
\end{equation}
with $w_r\in \G_{N-M}$ $(1\le r\le s)$. We remark here that since $p\nmid k$, we have 
$k\in \mathcal{S}(\calK)'$. Moreover, since $p^v k\not \in \mathcal{S}(\calK)$ for any 
$v\in \Z^+$, the equation of degree $k$ in \eqref{w7} is independent of the remaining 
equations of degree $j\in \mathcal{S}(\calK)\setminus\{k\}$. Therefore, we deduce from 
\eqref{w6} that
\[
\Psi =\usum_{l\in \mathcal{L}_r}\biggl| \usum_{\substack{h_i \in \G_{i(N-M)}\\ 
i\in \mathcal{S}(\calK)'}}b({\bf h})e\Big(\alpha_kh_k l^k+
\usum_{j\in \mathcal{S}(\calK)\setminus \{k\}}\gamma_j h_jl^j\Big)\biggr|^2.
\]
On recalling the definition \eqref{w2} of $\kappa(\mathcal K)$, we have
\[
\sum_{i\in \mathcal S(\mathcal K)'\setminus \{k\}}i=\kappa(\mathcal K)-k.
\]
Thus, we may conclude via Cauchy's inequality that
\begin{equation}\label{w8}
\Psi\le (q^{N-M})^{\kappa(\mathcal K)-k}\usum_{\substack{h_i \in \G_{i(N-M)}\\ 
i\in \mathcal{S}(\calK)'\setminus\{k\}}}\usum_{l \in \mathcal{L}_r}
\biggl| \usum_{h_k\in \G_{k(N-M)}} b({\bf h})e(\alpha_kh_k l^k)\biggr|^2.
\end{equation}

\par Since $p\nmid k$, it follows from Theorem \ref{largesieve} and Lemma \ref{spacing} 
that
\[
\usum_{l \in \mathcal{L}_r}\biggl| \usum_{h_k\in \G_{k(N-M)}}b({\bf h})
e(\alpha_kh_k l^k)\biggr|^2\le \big( |g|+q^{k(N-M)}\big) \usum_{h_k\in\G_{k(N-M)}}
|b({\bf h})|^2.
\]
Furthermore, by considering the underlying equations and recalling our assumption that 
$s\ge \psi \phi +\psi$, it follows from Theorem \ref{vmt} that there exists a constant 
$C_2=C_2(s;\calK;\epsilon;q)>0$ having the property that
\[
\usum_{\substack{h_i \in \G_{i(N-M)}\\ i\in \mathcal{S}(\calK)'\setminus\{k\}}}
\usum_{h_k\in \G_{k(N-M)}}|b({\bf h})|^2\le J_s(\mathcal{S}(\calK);N-M)
\le C_2(q^{N-M})^{2s-\kappa (\mathcal K)+\epsilon}.
\]
Since $\ord g\le kM$ and $2M\le N$, we may combine these estimates within \eqref{w8} to 
obtain the bound
\begin{align*}
\Psi &\le C_2(q^{N-M})^{2s-k+\epsilon}\big( |g|+q^{k(N-M)}\big) \\
&\le 2C_2(q^{N-M})^{2s+\epsilon}.
\end{align*}
We substitute this bound into \eqref{w5}, again noting that $\ord g\le kM$, to obtain the 
estimate
\[
\biggl| \usum_{x\in \GN}e(f(x))\biggr|\le 2q^N \big( 2C_1C_2M(q^M)^{-1}
(q^{kM})^\epsilon \big(q^{N-M})^\epsilon \big)^{1/(2s)}.
\]
Therefore, there exists a constant $C_3=C_3(s;\calK;\epsilon;q)>0$ such that for values of 
$M$ sufficiently large in terms of $\calK$, $\epsilon$ and $q$, one has
\[
\biggl| \usum_{x\in \GN}e(f(x))\biggr| \le q^N \big( C_3(q^M)^{-1}(q^N)^{k\epsilon}
\big)^{1/(2s)}.
\]

\par We now make the specific choice 
\begin{equation}\label{w10}
M=\lfloor \log_q C_3+kN\epsilon +2s\eta+1 \rfloor .
\end{equation}
Then it follows that 
\[
\biggl| \usum_{x \in \GN}e(f(x))\biggr| <q^{N-\eta },
\]
which contradicts the lower bound assumed in the statement of Theorem \ref{th:main1}. In 
view of the assumed bounds \eqref{w9}, this contradiction forces us to conclude that there 
exist $a\in \F_q[t]$ and monic $g\in \F_q[t]$ such that 
\[
\ord(g\alpha_k-a)<-kN+M \quad \text{and} \quad \ord g \le M.
\] 
Take $s=\psi\phi + \psi$, and then put $c=1/(8s)$ and $C=2s$. By assuming that 
$\epsilon <1/(4(k+1))$, we see that the requirement $2M\leq N$ is satisfied when 
$0<\eta \leq cN$, provided that $N$ is sufficiently large in terms of $\calK$, $\epsilon$ and 
$q$. We note that $c$ and $C$ are then constants depending only on $\calK$ and $q$. 
Moreover, when $N$ is sufficiently large, it follows from \eqref{w10} that 
\[
M\leq N(k+1)\epsilon +2s\eta\le N(k+1)\epsilon +C\eta.
\] 
Since $\epsilon>0$ is arbitrary, the conclusion of Theorem \ref{th:main1} follows. 
\end{proof}

\section{Extending the Weyl-type estimate to other coefficients} \label{sec:weyl2}
In this section, we extend Theorem \ref{th:main1} to indices which are not maximal. In 
preparation for the statement of this conclusion, we recall the definition \eqref{eq:k-star} 
of $\calK^*$.

\begin{thm}\label{th:main2}
Fix $q$ and a finite set $\calK \subset \Z^+$, and consider an integer $k\in \calK^*$. 
There exist positive constants $c_k$ and $C_k$, depending only on $k$, $\calK$ and $q$, 
such that the following holds. Let $\epsilon>0$ and let $N$ be sufficiently large in terms of 
$\calK$, $\epsilon$ and $q$. Suppose that $f(u)=\sum_{r\in \calK\cup\{0\}}\alpha_r u^{r}$ 
is a polynomial with coefficients in $\K_\infty$ satisfying the bound
\[
\biggl| \sum_{x\in \GN}e(f(x))\biggr| \geq q^{N-\eta},
\]
for some positive number $\eta$ with $\eta \le c_kN$. Then, there exist $a_k\in \Fq[t]$ and 
monic $g_k\in \Fq[t]$ such that  
\[
\ord (g_k\alpha_k-a_k)<-kN+\epsilon N+C_k\eta \quad \textrm{and}\quad 
\ord g_k \leq \epsilon N+C_k\eta.
\]
\end{thm}

\begin{proof}
Without loss of generality, we can assume that $\alpha_0=0$. We prove this theorem by 
downward induction on $k\in \calK^*$ with respect to the partial order $\preceq_{p}$. If 
$k$ is maximal in $\calK$, then the conclusion is immediate from Theorem \ref{th:main1}. 
Suppose that the conclusion of the theorem has been established for any $h\in \calK^*$ 
with $k\preceq_p h$ and $h\neq k$. Define
\begin{equation}\label{k0k1}
\calH_0=\{ h\in \calK : \text{$k\preceq_p h$ and $h\ne k$} \}\quad \text{and}\quad 
\calH_1=\calK\setminus \calH_0.
\end{equation}
Then it follows from Lemma \ref{lem:shadow}(c) that $\calH_0\subset \calK^*$. For 
$h\in \calH_0$, let $c_h$ and $C_h$ be the positive constants whose existence is assured 
by the inductive hypothesis, as a consequence of the conclusion of Theorem \ref{th:main2}. 
Let 
\[
c=\min\big\{c_h: h\in\calH_0\big\}\quad \textup{and}\qquad C=\sum_{h\in \calH_0}C_h.
\] 
Suppose that for some positive number $\eta$ with $\eta \leq cN$, one has 
\begin{equation}\label{upperbound}
\biggl| \sum_{x\in \GN}e(f(x))\biggr| \geq q^{N-\eta}.
\end{equation}
Let $\epsilon>0$ be arbitrary, and let $N$ be sufficiently large in terms of $\calK$, 
$\epsilon $ and $q$. Then, by the inductive hypothesis, for any $h \in \calH_0$ there exist 
$a_h\in \Fq[t]$ and monic $g_h\in \Fq[t]$ such that 
\[
\ord(g_h\alpha_h-a_h)<-hN+|\calH_0|^{-1}\epsilon N+C_h\eta \quad \text{and} \quad 
\ord g_h\le |\calH_0|^{-1}\epsilon N+C_h\eta  .
\]
Define 
\[
g=\prod_{h \in \calH_0}g_h\quad \textup{and}\quad b_h=a_h\prod_{j\in \calH_0
\setminus \{ h\} }g_j.
\] 
Then $g$ is monic and we have
\begin{equation}\label{w11}
\ord(g\alpha_h-b_h)<-hN+\epsilon N+C\eta \quad \text{and}\quad \ord g\le 
\epsilon N+C\eta .
\end{equation}

\par Consider a positive integer $M$ with $M<N-\ord g$. We rewrite the set $\GN$ first as 
a union of arithmetic progressions modulo $g$, and then subdivide these arithmetic 
progressions into subprogressions of appropriately small length. Thus we obtain
\begin{align*}
\GN&=\big\{ gv + w:\text{$v\in \G_{N-\ord g}$ and $w\in \G_{\ord g}$}\big\} \\
&=\big\{ g(t^Mz+y)+w:\text{$z\in \G_{N-M-\ord g}$, $y\in \G_M$ and 
$w\in \G_{\ord g}$}\big\} . 
\end{align*}
For each $z\in \G_{N-M-\ord g}$ and $w\in \G_{\ord g}$, write $s=gt^Mz+w$. Then 
$\ord s<N$ and we see that the set $\GN$ can be partitioned into $q^{N-M}$ blocks of the 
form
\[
\mathcal{B}_s=\big\{ gy+s: y\in \G_M\big\}.
\]
Then it follows from the lower bound (\ref{upperbound}) that there exists a block 
$\mathcal{B}_s$ such that
\begin{equation}\label{th2-1}
\biggl| \sum_{x \in \mathcal{B}_s}e(f(x))\biggr| =\biggl| \sum_{y\in \GM}e(f(gy+s))\biggr| 
\geq q^{N-\eta}\big(q^{N-M}\big)^{-1}=q^{M-\eta }.
\end{equation}
By reference to \eqref{k0k1}, we see that
\[
\biggl|\sum_{y \in \GM}e(f(gy+s))\biggr| =\biggl| \sum_{y\in \GM}e\biggl( 
\sum_{h \in \calH_0}\alpha_h (gy+s)^h +\sum_{h\in \calH_1}\alpha_h (gy+s)^h\biggr) 
\biggr| .
\]
Write $\beta_h=\alpha_h-b_h/g$ $(h\in \calH_0)$. Also, note that
\[
e\biggl( \sum_{h\in \calH_0}\alpha_hs^h\biggr)
\]
is a constant independent of $y$, and
\[
e\biggl( \sum_{h\in \calH_0}\frac{b_h}{g}\left( (gy+s)^h-s^h\right) \biggr)=1.
\]
Then we see that
\begin{equation}\label{th2-2}
\biggl|\sum_{y \in \GM}e(f(gy+s))\biggr| =\biggl| \sum_{y\in \GM}e\biggl( 
\sum_{h\in \calH_0}\beta_h\big( (gy+s)^h-s^h\big)+\sum_{h\in \calH_1}\alpha_h 
(gy+s)^h\biggr) \biggr| . 
\end{equation}

\par For any $y \in \GM$ and $h \in \calH_0$, we have 
\begin{align*}
\ord \big( (gy+s)^h-s^h\big) &\leq \ord (gy)+ (h-1)\cdot \max \big\{ \ord (gy),\ord s\big\} 
\\
&<\ord g+M+(h-1)N.
\end{align*}
It therefore follows from \eqref{w11} that 
\begin{align*} 
\ord \big( \beta_h\big( (gy+s)^h-s^h\big) \big)&<(-hN+\epsilon N+C\eta -\ord g)+
(\ord g+M+(h-1)N)\\
&=\epsilon N+C\eta +M-N.
\end{align*}
We now make the specific choice 
\[
M=\lfloor (1-\epsilon )N-C\eta-1\rfloor .
\]
Then it follows that 
\[
\epsilon N+C\eta +M-N\le -1,
\] 
and hence 
\[
\ord \big( \beta_h \big( (gy+s)^h-s^h\big) \big)< -1.
\]
Therefore, we have 
\begin{equation}\label{th2-3}
e\biggl( \sum_{h\in \calH_0} \beta_h\big( (gy+s)^h-s^h\big) +\sum_{h\in \calH_1}
\alpha_h(gy+s)^h\biggr) =e\biggl( \sum_{h\in \calH_1}\alpha_h (gy+s)^h\biggr).
\end{equation}
Combining (\ref{th2-1}), (\ref{th2-2}) and (\ref{th2-3}), we obtain the lower bound
\begin{equation}\label{w13}
\biggl| \sum_{y\in \GM}e\biggl( \sum_{h\in \calH_1}\alpha_h (gy+s)^h\biggr)\biggr|\geq 
q^{M-\eta}.
\end{equation}
We note here that from \eqref{w11} we have $\ord g\le \epsilon N+C\eta$, and thus for 
$N$ sufficiently large, the above choice of $M$ satisfies $0<M<N-\ord g$.\par

In view of the definition \eqref{w1}, we have
\begin{equation}\label{w12}
\sum_{h\in \calH_1}\alpha_h (gy+s)^h=\sum_{j\in \calS (\calH_1)\cup\{ 0\}}\gamma_jy^j,
\end{equation}
for suitable coefficients $\gamma_j=\gamma_j({\boldsymbol \alpha},g,s)\in \K_\infty$. 
Since $k\in \calK^*$ is maximal in $\calH_1$, it follows from Lemma \ref{lem:shadow} that 
$k$ is maximal in $\mathcal{S}(\calH_1)$ and $k\in \mathcal{S}(\calH_1)^*$. Furthermore, 
the coefficient of $y^k$ in the polynomial on the left hand side of \eqref{w12} is 
$\alpha_k g^k$. Note also that we may suppose the parameter $M$ to be sufficiently large 
in terms of $\calK$, $\epsilon$ and $q$. Thus, by Theorem \ref{th:main1}, there exist 
positive constants $d_k$ and $D_k$ having the property that whenever the lower bound 
\eqref{w13} holds for some positive number $\eta$ with $\eta\le d_kM$, then there exist 
$\widetilde{a}_k\in \F_q[t]$ and monic $\widetilde{g}_k\in \Fq[t]$ such that
\[
\ord (\widetilde{g}_k\alpha_kg^k-\widetilde{a}_k)<-kM+\epsilon M+D_k\eta \quad
\textup{and} \quad  \ord \widetilde{g}_k\leq \epsilon M+D_k\eta .
\] 
Let $g_k=\widetilde{g}_kg^k$ and $a_k=\widetilde{a}_k$. Since 
$(1-\epsilon )N-C\eta-2<M\le N$, for $N$ sufficiently large, we have 
\begin{align*}
\ord (g_k\alpha _k-a_k)&<-k\big( (1-\epsilon )N -C\eta-2\big) +\epsilon N+D_k\eta \\
&<-kN+\epsilon (k+2)N+\left( kC+D_k\right)\eta 
\end{align*}
and, on recalling \eqref{w11}, 
\[ 
\ord g_k\le (\epsilon M+D_k\eta )+k(\epsilon N+C\eta)\le \epsilon (k+1)N+(kC+D_k)\eta.
\]
Since $\epsilon>0$ is arbitrary, the conclusion of Theorem \ref{th:main2} follows for $k$ by 
taking $c_k=\min \{ c,d_k\}$ and $C_k=kC+D_k$. This confirms the inductive step, and 
thus the proof of the theorem is complete.  
\end{proof}

One can extend Theorem \ref{th:main2} to indices that are not in $\calK^*$. Recall the 
definition \eqref{eq:ktilde} of $\tK$. Then by induction on $n$, one can apply the method of 
the proof of Theorem \ref{th:main2} to obtain the following conclusion.

\begin{prop} \label{prop:varmain2}
Fix $q$ and a finite set $\calK \subset \Z^+$. There exist positive constants $c$ and $C$, 
depending only on $\calK$ and $q$, such that the following holds. Let $\epsilon>0$ and let 
$N$ be sufficiently large in terms of $\calK$, $\epsilon$ and $q$. Suppose that 
$f(u)=\sum_{r\in \calK\cup\{0\}}\alpha_r u^{r}$ is a polynomial with coefficients in 
$\K_\infty$ satisfying the bound
\[
\biggl| \sum_{x\in \GN} e(f(x))\biggr| \geq q^{N-\eta },
\]
for some positive number $\eta$ with $\eta\le cN$. Then, for any $k\in \tK$, there exist 
$a_k\in \Fq[t]$ and monic $g_k\in \Fq[t]$ such that
\[
\ord (g_k\alpha_k-a_k)<-kN+\epsilon N+C\eta \quad \textrm{and} \quad 
\ord g_k\leq \epsilon N+C\eta.
\]
\end{prop}

It seems that there is no simple description of the set $\widetilde{\calK}$. In many cases, 
it is apparent that $\widetilde{\calK}$ is larger than $\calK^*$. For example, if $p>3$ and 
$\calK=\{1,3, 3p+1\}$ (as in the first case of Example \ref{example-3}), then
\[
\calS (\calK)=\{1,2,3,p,p+1,2p,2p+1,3p,3p+1\},
\]
and so $\calK^*=\{3p+1\}$. Meanwhile, since $\calK_1=\{1,3\}$, one finds that 
$\calK_1^*=\{1,3\}$, and since $\calS(\calK_1)=\{1,2,3\}$, it follows from 
\eqref{eq:ktilde} that $\widetilde{\calK}=\calK$. More generally, if $(k,p)=1$ for any 
$k\in \calK$, then it can be proved by induction that $\tK=\calK$. On the other hand, if  
$p>3$ and $\calK = \{3, 4p\}$ (as in the second case of Example \ref{example-3}), 
then
\[
\calS(\calK)=\{1,2,3,p,2p,3p,4p\},
\]
and hence $\calK^*=\emptyset$. Thus we find that in this case, one has $\tK=\emptyset$. 
Therefore, we cannot go as far as proving Conjecture \ref{conj} by using this method.

\section{Equidistribution of polynomial sequences}\label{sec:equidistribution}
In this section, we first prove the equidistribution result recorded in Theorem 
\ref{th:main3}, and then discuss a variant of this theorem. The following lemma is essential 
for our proof of Theorem \ref{th:main3}. We again recall the set of exponents $\calK^*$ 
defined in \eqref{eq:k-star}.

\begin{lem}\label{pre-main3}
Fix $q$ and a finite set $\calK \subset \Z^+$. Let 
$f(u)=\sum_{r\in \calK\cup\{0\}}\alpha_r u^{r}$ be a polynomial with coefficients in 
$\K_\infty$. For $k\in \calK^*$, suppose that $k$ is maximal in $\calK$ and $\alpha_k$ is 
irrational. Then, for any fixed $\eta >0$, there exists $N_0\in \Z^+$ such that, for any 
$s\in \Fq[t]$, we have
\[
\biggl| \sum_{y\in \G_{N_0}}e(f(y+s))\biggr| <q^{N_0-\eta}.
\]
\end{lem}

\begin{proof}
By way of deriving a contradiction, suppose that $\eta>0$, and that for any $N\in \Z^+$, 
there exists $s_N \in \Fq[t]$ such that
\begin{equation}\label{w14}
\biggl| \sum_{y \in \G_{N}} e(f(y+s_N)) \biggr| \geq q^{N-\eta}.
\end{equation}
We note that for each $s\in \Fq[t]$, the only monomials $y^r$ having non-zero coefficient 
in the expansion of $f(y+s)$ are those with $r\in \calS(\calK)$. Since $k\in \calK^*$ is 
maximal in $\calK$, it follows from Lemma \ref{lem:shadow} that $k$ is maximal in 
$\calS(\calK)$ and further that $k\in \calS(\calK)^*$. Moreover, the coefficient of $y^k$ in
$f(y+s)$ is $\alpha_k$. Applying Theorem \ref{th:main1} with $\epsilon=1/3$, we find that 
there exists a constant $C>0$ such that, for $N$ sufficiently large in terms of $\calK$ and 
$q$, there exist $a\in \Fq[t]$ and monic $g\in \Fq[t]$ having the property that 
\[
\ord (g\alpha_k-a)\leq -kN+N/3+C\eta \quad \textrm{and} \quad \ord g<N/3+C\eta .
\]
For each sufficiently large $M\in \Z^+$, we apply these inequalities with 
$N=\lfloor 3(M-C\eta)\rfloor$. Thus, we have 
\[
\ord (g\alpha_k -a)\leq -(3k-1)M+(3kC\eta +k-1/3)\leq -3M/2\quad \textup{and}\quad  
\ord g<M. 
\]
Since these inequalities hold for all sufficiently large $M\in \Z^+$, we deduce from Lemma 
\ref{lem:diophantine} that $\alpha_k$ is rational, contradicting the hypothesis that 
$\alpha_k$ is irrational. Consequently, the assumed lower bound \eqref{w14} is untenable, 
and the conclusion of the lemma follows.
\end{proof}

We are now equipped for the proof of Theorem \ref{th:main3}. 

\begin{proof}[Proof of Theorem \ref{th:main3}]
It is apparent that there is no loss of generality in assuming that $\alpha_0=0$. Let 
$k\in \calK^*$ and suppose that $\alpha_k$ is irrational. We prove Theorem \ref{th:main3} 
by downward induction on $k$ with respect to the partial order $\preceq_{p}$. Suppose 
first that $k$ is maximal in $\calK$ and $\eta>0$. Let $N_0$ be the natural number 
provided in the conclusion of Lemma \ref{pre-main3}. For any $N\geq N_0$, we can 
partition the set $\GN$ into $q^{N-N_0}$ blocks of the form 
\[
\mathcal{B}_s = \left\{y+s: y\in \G_{N_0}\right\},
\] 
where $s=t^{N_0}z$ for some $z\in \G_{N-N_0}$. Therefore, it follows from Lemma 
\ref{pre-main3} that 
\[
\biggl| \sum_{x\in \G_{N}}e(f(x))\biggr| \le q^{N-N_0}\sup_{s\in \Fq[t]}\biggl| 
\sum_{y\in \G_{N_0}}e(f(y+s))\biggr|<q^{N-N_0}q^{N_0-\eta}=q^{N- \eta}.
\]
Since $\eta>0$ is arbitrary, it follows that
\[
\lim_{N\rightarrow \infty}\frac{1}{q^N}\biggl| \sum_{x\in \G_{N}}e(f(x))\biggr| =0.
\]
We note that for any $m\in \F_q[t]\setminus \{ 0\}$, this relation holds with $f$ replaced by 
$mf$, where $mf$ is the polynomial
\[
mf(u)=\sum_{r\in \calK\cup \{0\}}m\alpha_r u^r.
\]
By reference to Theorem \ref{th:weyl}, we therefore conclude that Theorem \ref{th:main3} 
holds in the special case in which $k$ is maximal in $\calK$.\par 

Suppose next that the theorem is established for any $h\in \calK^*$ with $k\preceq_p h$ 
and $h\neq k$. We define $\calH_0$ and $\calH_1$ as in \eqref{k0k1}. Note that, should 
there exist $h\in \calH_0$ for which $\alpha_h$ is irrational, then Theorem \ref{th:main3} 
follows from the inductive hypothesis. Therefore, it suffices to consider the situation in 
which all of the coefficients $\alpha_h$ $(h\in \calH_0)$ are rational. Let $g$ be the 
common denominator of the coefficients $\alpha_h$ for $h\in \calK_0$. Then for any 
$s\in \Fq[t]$ and $M \in \Z^+$, we have 
\begin{align*}
\biggl| \sum_{y\in \GM}e(f(gy + s))\biggr| &=\biggl| \sum_{y\in \GM}e\biggl( 
\sum_{h\in \calK}\alpha_h(gy+s)^h\biggr) \biggr| \\
&=\biggl| \sum_{y\in \GM}e\biggl( \sum_{h\in \calH_0}\alpha_h\biggl( (gy+s)^h-s^h\biggr) 
+\sum_{h\in \calH_1}\alpha_h(gy+s)^h\biggr) \biggr| .
\end{align*} 
Here, we have made use of the observation that
\[
e\biggl( \sum_{h\in \calH_0}\alpha_h(-s^h)\biggr)
\]
is a unimodular constant independent of $y$. Since the definition of $g$ implies that 
$g\alpha_h\in \Fq[t]$ for each $h\in \calH_0$, we have 
\[
e\biggl( \sum_{h\in \calH_0}\alpha_h \biggl( (gy+s)^h -s^h\biggr) \biggr) =1.
\]
It follows that 
\begin{equation}\label{egy}
\biggl| \sum_{y\in \GM}e(f(gy+s))\biggr| =\biggl| \sum_{y\in \GM}e\biggl( 
\sum_{h\in \calH_1}\alpha_h (gy+s)^h\biggr) \biggr| .
\end{equation}
Given $N\in \Z^+$ with $N>\ord g$, we define the integer $M\in \Z^+$ by putting 
$M=N-\ord g$. Then we can partition the set $\G_N$ into $q^{N-M}$ blocks of the form 
\[
\mathcal{B}_s=\left\{ gy+s: y\in \G_M \right\},
\]
where $s\in \G_{\ord g}$. We now deduce from from \eqref{egy} that
\begin{align}
\biggl| \sum_{x\in \GN}e(f(x))\biggr| &\le q^{N-M}\max_{s\in \G_{\ord g}}\biggl| 
\sum_{y\in \GM}e(f(gy+s))\biggr| \notag \\
&=q^{N-M}\max_{s\in \G_{\ord g}}\biggl| \sum_{y\in \GM}e\biggl( \sum_{h \in \calH_1} 
\alpha_h (gy+s)^h \biggl) \biggr| . \label{main3-equal}
\end{align}

We observe that for each $s\in \Fq[t]$, the only monomials $y^r$ having non-zero 
coefficient in the expansion of
\begin{equation}\label{w15}
\sum_{h\in \calH_1}\alpha_h(gy+s)^h
\end{equation}
are those with $r\in \calS(\calH_1)$. Since $k\in \calK^*$ is maximal in $\calH_1$, we 
discern from Lemma \ref{lem:shadow} that $k$ is maximal in $\calS(\calH_1)$ and 
$k\in \calS(\calH_1)^*$. Furthermore, the coefficient of $y^k$ in the polynomial 
\eqref{w15} is $\alpha_k g^k$, which is irrational since $\alpha_k$ is irrational. We are now 
in the situation already handled in the first part of the proof, and thus, we have 
\[
\lim_{M\to \infty}\frac{1}{q^M}\biggl| \sum_{y\in \GM}e\biggl( \sum_{h\in \calH_1} 
\alpha_h (gy+s)^h\biggr) \biggr| =0.
\]
Then it follows from (\ref{main3-equal}) that 
\[
\lim_{N\rightarrow \infty}\frac{1}{q^N}\biggl| \sum_{x\in \G_{N}} e(f(x))\biggr| =0.
\]
We again note that for any $m\in \F_q[t]\setminus \{ 0\}$, this relation remains valid with 
$f$ replaced by $mf$, and thus Theorem \ref{th:weyl} shows the sequence 
$(f(x))_{x\in \Fq[t]}$ to be equidistributed in $\T$. This confirms the inductive step, and 
thus the proof of the theorem is complete.
\end{proof}

By an observation similar to the one made following the proof of Theorem \ref{th:main2}, 
one can apply the method of the proof of Theorem \ref {th:main3} to obtain the following 
result. Here, once again, we recall the definition \eqref{eq:ktilde} of the set of exponents 
$\tK$.

\begin{prop}\label{prop:varmain3}
Fix $q$ and a finite set $\calK \subset \Z^+$. Let 
$f(u)=\sum_{r\in \calK \cup \{0\}}\alpha_r u^{r}$ be a polynomial with coefficients in 
$\K_\infty$. Suppose that $\alpha_k$ is irrational for some $k\in \tK$. Then the sequence 
$(f(x))_{x \in \Fq[t]}$ is equidistributed in $\T$.
\end{prop}

Of notable significance in this conclusion is the situation in which $(k,p)=1$ for all 
$k\in \calK$, for then we have $\tK=\calK$. Using the latter observation, we now show that 
the above proposition implies Conjecture \ref{conj} in the special case $q=p$. For the rest 
of this section, we assume that $q=p$.\par

Let $T:\K_\infty \rightarrow \T$ be the map defined in (\ref{eq:t}). Using the fact that 
$a^p=a$ for any $a\in \F_p$, one can show that for any $x \in \Fp[t]$, one has
\[
e\left( \alpha x^p\right) =e\left( T(\alpha) x \right).
\]
Therefore, for any $x\in \Fp[t]$ and $v\in \Z^+\cup\{0\}$, we have 
\begin{equation}\label{eq:t2}
e\left( \alpha x^{p^v}\right) =e\left( T^v(\alpha)x\right),
\end{equation}
where $T^v$ is the $v$-fold composition of $T$. Let
\[
f(u)=\sum_{r\in \calK\cup \{0\}}\alpha_r u^r\in \K_\infty[u],
\]
and let 
\begin{equation}\label{cali}
\calI =\{ k\in \Z^+:\text{$(k,p)=1$ and $p^v k\in \calK$ for some $v\in \Z^+\cup\{0\}$}
\}.
\end{equation}
For each $k\in \calI$, define 
\begin{equation}\label{eq:s}
S_k(f)=\sum_{\substack{v\geq 0\\p^v k\in \calK}}T^v(\alpha_{p^v k}).
\end{equation}
Then it follows from (\ref{eq:t2}) that for any $x\in \Fp[t]$, one has
\begin{equation}\label{esf}
e\left( f(x)\right) =e\biggl( \sum_{k\in \calI}S_k(f)x^k+\alpha_0\biggr) .
\end{equation}
Since $(k,p)=1$ for any $k\in \calI$, we have $\widetilde{\calI}=\calI$. Let 
$m\in \F_p[t]\setminus \{ 0\}$. Then Proposition \ref{prop:varmain3} shows that whenever 
there exists $k\in \calI$ such that $S_k(mf)$ is irrational, one has
\begin{equation}\label{eq:limit}
\lim_{N\rightarrow \infty}\frac{1}{q^N}\biggl| \sum_{x\in \GN}e(mf(x))\biggr| 
=\lim_{N\rightarrow \infty}\frac{1}{q^N}\biggl| \sum_{x\in \GN}
e\biggl( \sum_{k\in \calI}S_k(mf)x^k+m\alpha_0\biggr) \biggr| =0. 
\end{equation}
Therefore, on making use of Theorem \ref{th:weyl}, we may conclude as follows. 

\begin{cor}\label{cor:q=p}
Fix $q=p$ and a finite set $\calK \subset \Z^+$. Let 
$f(u)=\sum_{r\in \calK\cup\{0\}}\alpha_r u^{r}$ be a polynomial with coefficients in 
$\K_\infty$. Suppose that the polynomial $f$ satisfies the property that for some 
$k\in \calI$, we have
\begin{equation}\label{eq:irrational}
\text{$S_k(mf)$ is irrational for any $m\in \Fp[t]\setminus \{0\}$}. 
\end{equation}
Then the sequence $(f(x))_{x\in \Fp[t]}$ is equidistributed in $\T$.
\end{cor} 

We remark that since the map $T$ does not commute with multiplication by $m$, the 
condition (\ref{eq:irrational}) may not be described in simpler terms. This condition might 
also be unnecessary for the equidistribution of $(f(x))_{x \in \Fp[t]}$. Regardless of these 
observations, suppose that $k\in \calK$ and $p^v k\not \in \calK$ for any $v\in \Z^+$. 
Then $S_k(f)=\alpha_k$ and $S_k(mf)=m\alpha_k$ for any $m\in \Fp[t]\setminus\{0\}$. 
Therefore, should $\alpha_k$ be irrational, then the condition (\ref{eq:irrational}) is 
satisfied. This simple observation establishes Conjecture \ref{conj} in the special case 
$q=p$. We can formulate this conclusion more precisely in the following corollary. 

\begin{cor}\label{cor:q=pw}
Fix $q=p$ and a finite set $\calK \subset \Z^+$. Let 
$f(u)=\sum_{r\in \calK\cup\{0\}}\alpha_r u^{r}$ be a polynomial with coefficients in 
$\K_\infty$.~Suppose that $\alpha_k$ is irrational for some $k\in \calK$ satisfying 
$p\nmid k$ and furthermore $p^v k\not \in \calK$ for any $v\in \Z^+$. Then the sequence 
$(f(x))_{x\in \Fp[t]}$ is equidistributed in $\T$.
\end{cor} 

\section{Van der Corput and intersective sets in $\Fq[t]$} \label{sec:vdc}

\subsection{Background and statement of results.} We define the {\it upper density} 
${\overline d}(\mathcal{A})$ of a set $\mathcal{A}\subset \Z^+$ by means of the relation
\[
\overline{d}(\mathcal{A})=\limsup_{N\rightarrow \infty}\frac{\text{card}(
\mathcal{A}\cap \{1, \ldots, N\})}{N}.
\] 
We say that $\mathcal{A}$ is \textit{dense} if $\overline{d}(\mathcal{A})>0$. A set 
$\mathcal{H}\subset \Z^+$ is called \textit{intersective} if, for any dense subset 
$\mathcal{A}\subset \Z^+$, there exist $a,a'\in \mathcal{A}$ such that 
$a-a'\in \mathcal{H}$. Thus, the set $\mathcal{H}$ is intersective if for any dense subset 
$\mathcal{A}$ of positive integers, one has 
$\mathcal{H}\cap (\mathcal{A}-\mathcal{A})\ne \emptyset $. In the late 1970s, 
S\'{a}rk\"{o}zy \cite{sarkozy1} and Furstenberg \cite{f2} proved independently that the set 
$\{ n^2:n \in \Z^+\}$ is intersective. Their proofs make use of the circle method and 
ergodic theory, respectively. S\'{a}rk\"{o}zy went on to prove that the sets 
$\{ n^2-1:n\in \Z^+\setminus\{1\}\}$ and $\{p-1:\text{$p\in \Z$ is prime}\}$ are also 
intersective (see \cite{sarkozy3}). We refer the reader to a survey paper of the first author 
\cite{le} for results and open problems regarding intersective sets.\par

In a seemingly unrelated context, motivated by van der Corput's difference theorem, Kamae 
and Mend\`{e}s France \cite{km} made the following definition. A set 
$\mathcal{H}\subset \Z^+$ is said to be \textit{van der Corput} if the sequence 
$(a_n)_{n=1}^\infty$ is equidistributed $(\mod 1)$ whenever the sequence 
$(a_{n+h}-a_n )_{n=1}^\infty$ is equidistributed $(\mod 1)$ for each $h\in \mathcal{H}$. 
Therefore, it follows from van der Corput's difference theorem that $\Z^+$ is van der
Corput. However, there are sparser sets which are van der Corput. In \cite{km}, 
Kamae and Mend\`{e}s France proved that any van der Corput set is intersective. Their 
result gives another approach to intersective sets. The converse of their theorem is not 
true. In \cite{bourgain}, Bourgain constructed a set that is intersective but not van der 
Corput.\par

Let $\Phi(u)\in \Z[u]$ and consider the set $\{ \Phi(n): n\in \Z \}\cap \Z^+$. We note that 
for any $g\in \Z^+$, the set of all multiples of $g$ is dense. Therefore, if the set 
$\{\Phi(n): n \in \Z\}\cap \Z^+$ is van der Corput (and hence intersective), then $g$ divides 
$\Phi(n)$ for some $n\in \Z$. The following result of Kamae and Mend\`{e}s France 
\cite{km} shows that the divisibility condition is not only necessary, but also sufficient. 

\begin{prop}\label{prop:intersective} 
Let $\Phi(u)\in \Z[u]\setminus\{0\}$, and suppose that $\Phi$ has a root $(\mod \,g)$ for 
any $g\in \Z^+$. Then the set $\left\{ \Phi(n): n\in \Z \right\} \cap \Z^+$ is van der Corput 
(and hence intersective) whenever it is infinite. 
\end{prop}

Notice that these notions of intersective and van der Corput sets, and the concommitant 
conclusions, extend readily to the situation that $\calA\subset \Z$ and 
$\calH\subset \Z\setminus \{0\}$. Given the similarity of $\Z$ and $\Fq[t]$, it is natural to 
study analogous notions in $\Fq[t]$. We define the {\it upper density} 
${\overline d}(\mathcal{A})$ of a set $\mathcal{A} \subset \Fq[t]$ by means of the 
relation
\[
\overline{d}(\mathcal{A})=\limsup_{N\rightarrow \infty}
\frac{\text{card}(\mathcal{A}\cap \GN)}{q^N}.
\]
We say a set $\mathcal{A}$ is {\it dense} if $\overline{d}(\mathcal{A})>0$. A set 
$\mathcal{H}\subset \Fq[t]\setminus \{0\}$ is called \textit{intersective} if, for any dense 
subset $\mathcal{A}\subset \F_q[t]$, we have 
$\mathcal{H}\cap (\mathcal{A}-\mathcal{A})\neq \emptyset$. A set 
$\mathcal{H}\subset \Fq[t]\setminus \{0\}$ is said to be \textit{van der Corput} if the 
sequence $(a_{x})_{x \in \Fq[t]}$ is equidistributed in $\T$ whenever the sequence 
$(a_{x+h}-a_{x})_{x \in \Fq[t]}$ is equidistributed in $\T$ for each $h \in \mathcal{H}$. 
Many characterizations of intersective and van der Corput sets carry over from $\Z$ to 
$\Fq[t]$, and we refer the reader to the Ph.D. thesis of the first author 
\cite[Chapter 2]{lethesis} for an exposition. In particular, in \cite[Theorem 2.3.5]{lethesis}, 
it was proved that any van der Corput set in $\Fq[t]$ is intersective. It is an interesting 
problem to construct a set in $\Fq[t]$ that is intersective but not van der 
Corput (Bourgain's construction in $\Z$ is very specific to the real numbers).\par 

We now consider explicit examples of intersective and van der Corput sets in $\Fq[t]$ that 
are of arithmetic interest, similar to the results of S\'{a}rk\"{o}zy and Furstenberg. In the 
work of the first two authors \cite{ll}, intersectivity is obtained, in a quantitative sense,  for 
the set $\left\{ x^2:x \in \Fq[t]\right\}\setminus\{0\}$. Furthermore, in joint work of the 
first author with Spencer \cite{ls}, intersectivity, in a quantitative sense, is also established 
for the set
\[
\left\{l+r:\text{$l\in \Fq[t]$, with $l$ monic and irreducible}\right\},
\]
for any fixed $r\in \F_q\setminus\{0\}$. Motivated by Proposition \ref{prop:intersective}, 
we formulate the following conjecture. 

\begin{conj}\label{conj:intersective}
For $\Phi(u)\in \Fq[t,u]\setminus\{0\}$, suppose that 
\begin{equation}\label{eq:intersective}
\text{for all $g\in \Fq[t]$, there exists $x\in \Fq[t]$ such that $\Phi(x)\equiv 0\mmod{g}$}. 
\end{equation}
Then the set $\left\{\Phi(x): x\in \Fq[t]\right\} \setminus\{0\}$ is van der Corput (and hence 
intersective). 
\end{conj}

Again, the divisibility condition is easily seen to be necessary. Quite surprisingly, this 
conjecture remains an open problem when the degree of $\Phi$ is greater than or equal to 
$p$. When $\Phi(0)=0$, it follows from the polynomial Szemer\'edi theorem for modules 
over countable integral domains, proved by Bergelson, Leibman and McCutcheon \cite{blm}, 
that the set $\left\{ \Phi(x): x\in \Fq[t]\right\} \setminus \{0\}$ is intersective. Recently, 
using the polynomial method of Croot, Lev and Pach \cite{clp}, it was shown by Green 
\cite{green} that this conjecture holds in a strong quantitative sense, under the condition 
that $\Phi(u)\in \Fq[u]$ and the number of roots of $\Phi(u)$ in $\Fq$ is coprime to $q$. 
The latter constraint was recently removed by Li and Sauermann \cite{LS2022}. We note 
that the condition \eqref{eq:intersective} is weaker than demanding that $\Phi(u)$ has a 
root in $\Fq[t]$. Indeed, by analogy with well-known examples over the rational integers, 
we observe that when $p>2$ and $a$ and $b$ are distinct irreducible polynomials of even 
degree in $\F_p[t]$ with $b$ a quadratic residue modulo $a$ (and hence also $a$ a 
quadratic residue modulo $b$), the polynomial $\Phi(u)=(u^2-a)(u^2-b)(u^2-ab)$ fails to 
have roots in $\F_p[t]$, yet nonetheless possesses solutions modulo $g$, for all 
$g\in \F_p[t]$. We direct the reader to Li \cite[Example 1]{li} and Yamagishi 
\cite[Appendix A]{yamagishi2} for examples of polynomials $\Phi$ satisfying 
\eqref{eq:intersective} but not having roots in $\Fq[t]$.\par

Equipped now with our equidistribution theorem, we make some progress in this section 
towards Conjecture \ref{conj:intersective}. In Section 6.3 we prove the following conclusion, 
which is slightly stronger than Theorem \ref{th:sarkozy}. Here, we recall the definition 
\eqref{eq:k-star} of the set of exponents $\calK^*$.   

\begin{thm}\label{th:vdc1}
Let $\calK$ be a finite set of positive integers, suppose that $a_r\in \Fq[t]$ for $r\in 
\calK\cup \{0\}$, and define
\[
\Phi(u)=\sum_{r\in \calK \cup \{0\}}a_r u^r .
\]
Suppose that $\Phi$ satisfies the condition \eqref{eq:intersective}. Suppose further that 
$a_k\neq 0$ for some $k\in \calK^*$. Then the set 
$\{ \Phi(x): x\in \Fq[t] \} \setminus\{0\}$ is van der Corput (and hence intersective).
\end{thm}

We remark that, as a direct consequence of Theorem \ref{th:vdc1}, one finds that 
Conjecture \ref{conj:intersective} holds whenever the degree of $\Phi$ is coprime to $p$. 
Moreover, in view of Proposition \ref{prop:varmain3}, the condition in the theorem requiring 
$a_k\neq 0$ for some $k\in \calK^*$ can be relaxed to one requiring only that $a_k\neq 0$ 
for some $k \in \tK$, where $\tK$ is defined as in (\ref{eq:ktilde}).\par

By assuming the stronger conditions $q=p$ and $\Phi(0)=0$, we obtain the following result 
in Section 6.3. 

\begin{thm}\label{th:vdc2}
Let $\Phi(u) \in \Fp[t,u]\setminus\{0\}$, and suppose that $\Phi(0)=0$. Then the set 
$\{ \Phi(x): x\in \Fp[t] \} \setminus \{0\}$ is van der Corput (and hence intersective).
\end{thm}

We remark here that the conclusion of Theorem \ref{th:main2} can be applied to prove 
intersectivity of the set $\{ \Phi(x): x \in \Fq[t] \} \setminus \{0\}$ in Theorem
\ref{th:vdc1} in a quantitative sense, in a manner similar to that employed in the proof of 
\cite [Theorem 3]{ll}. However, we opt to make use of Theorem \ref{th:main3} since the 
deduction is quicker, and the van der Corput property is a stronger notion than intersectivity. 

\subsection{Comparison with Bergelson-Leibman's result}\label{sec:bg}
Bergelson and Leibman \cite{bl} also applied their equidistribution result to study 
intersective sets in $\Fq[t]$. As such, our results in this section overlap with the conclusion 
of their Theorem 9.5, though they are not identical. Before proceeding with the proofs of 
Theorems \ref{th:vdc1} and \ref{th:vdc2}, we make a comparison between these theorems 
and \cite[Theorem 9.5]{bl}, which we rephrase below.

\begin{theorem*}[Bergelson-Leibman] Let $\Phi(u)\in \Fq[t,u]\setminus\{0\}$, and suppose 
that $\Phi(0)=0$. Then the set $\{ \Phi(x): x\in \Fq[t]\} \setminus \{0\}$ is intersective. 
Furthermore, the same conclusion holds provided that $\Phi$ satisfies the 
condition\footnote{See the remark in \cite[p. 949]{bl}, though there is a misprint in the 
definition of intersectivity therein.}$^,$\footnote{Just prior to the submission of this paper, 
Ackelsberg and Bergelson uploaded a paper \cite{AB2023} to the arXiv in which some 
correction and clarification concerning their notion of intersectivity over $\Fq[t]$ is made 
(see the first footnote on page 2 of \cite{AB2023} and the accompanying discussion). 
Nonetheless, at this time we remain unable to identify a source in the literature for a proof 
of Conjecture \ref{conj:intersective}, and it seems fair to describe the current status of the 
notion of intersectivity associated with this perspective as being in a state of flux.} that
\begin{equation}\label{eq:bl-intersective}
\text{for all subgroups $\Lambda$ of finite index in $(\Fq[t],+)$, there exists $x\in \Fq[t]$ 
such that $\Phi(x)\in \Lambda$}.
\end{equation}
\end{theorem*}

Bergelson and Leibman proved this theorem following the proof by Furstenberg \cite{f2} of 
S\'{a}rk\"{o}zy's theorem in $\Z$ (and in fact they proved a Khintchine-type theorem for 
single recurrence). On the other hand, our proofs of Theorems \ref{th:vdc1} and 
\ref{th:vdc2} follow the treatment of Kamae and Mend\`es France of van der Corput sets in 
$\Z$. Since in $\Fq[t]$, van der Corput sets and intersective sets are (conjecturally) two 
distinct notions, our own results and those of Bergelson and Leibman \cite[Theorem 9.5]{bl} 
do not imply each other.\par

The condition \eqref{eq:bl-intersective} is clearly necessary in order that the set 
$\{ \Phi(x): x\in \Fq[t]\} \setminus \{0\}$ be intersective. It is also easy to see that the 
condition \eqref{eq:bl-intersective} (an algebraic condition) implies \eqref{eq:intersective} 
(an arithmetic condition). We note, however, that there are plenty of subgroups of finite 
index in the additive group $\Fq[t]$ which are not of the shape $g\Fq[t]$ for any 
$g\in \Fq[t]$. For each irrational $\alpha\in \K_\infty$, an example of such a subgroup is the 
Bohr set consisting of all poynomials $x\in \Fq[t]$ satisfying the condition 
$\ord \{ \alpha x\}<-1$. We cannot help but wonder if the conditions 
\eqref{eq:bl-intersective} and \eqref{eq:intersective} are in fact the same condition. (This 
issue does not arise in $\Z$, since all subgroups of finite index of $\Z$ are of the form 
$a\Z$ for some $a \neq 0$.)

\begin{question}\label{q:intersective}
Does the condition \eqref{eq:intersective} imply \eqref{eq:bl-intersective}? In other words, 
as far as polynomials in $\Fq[t]$ are concerned, does ``meeting all subgroups of arithmetic 
nature'' imply ``meeting all subgroups of finite index''?
\end{question}
 
\subsection{The proofs of Theorems \ref{th:vdc1} and \ref{th:vdc2}}
Among the many characterizations of van der Corput sets in $\Fq[t]$, we will apply the 
following one found in \cite[Theorem 2.4.5 (2)]{lethesis}. Let $\mu$ be a finite 
non-negative measure on 
$\T$. We say that $\mu $ is {\it continuous} at $0$ if $\mu(\{0\})=0$. For any 
$h\in \F_q[t]$, the {\it Fourier transform} of $\mu$ is denoted by $\widehat{\mu }$ and 
defined by
\[
\widehat{\mu}(h)=\int_{\T} e(-\alpha h)\, {\rm d}\mu(\alpha).
\]
We say that $\widehat{\mu}$ {\it vanishes} on a set $\mathcal{H}\subset \F_q[t]$ if 
$\widehat{\mu}(h)=0$ for all $h \in \mathcal{H}$. 

\begin{thm}[Kamae \& Mend\`{e}s France, Ruzsa]\label{theoremw20}
A set $\mathcal{H}\subset \Fq[t]\setminus \{ 0 \}$ is van der Corput if and only if any finite 
measure $\mu$ on $\T$, with $\widehat{\mu}$ vanishing on $\mathcal{H}$, is continuous 
at $0$.
\end{thm}

We are now equipped to prove Theorems \ref{th:vdc1} and \ref{th:vdc2}.

\begin{proof}[Proof of Theorem \ref{th:vdc1}]
Suppose that $\Phi(u) = \sum_{k \in \calK\cup \{0\}}a_r u^r\in \Fq[t,u]$ has a root 
$(\mod{g})$ for any $g\in \F_q[t]\setminus\{0\}$. Suppose further that $a_k\neq 0$ for 
some $k\in \calK^*$. Let 
\begin{equation}\label{w16}
\mathcal{H}=\{ \Phi(x): x\in \Fq[t]\} \setminus \{0\}.
\end{equation}
Also, let $\alpha \in \T$ be irrational, and consider $s\in \Fq[t]$ and monic $g\in \Fq[t]$. By 
the orthogonality relation (\ref{eq:orthogonal2}), we have
\begin{align*}
\frac{1}{q^N}\sum_{\substack{x\in \GN\\ x\equiv s\, (\mod g)}} 
e\left( \alpha \Phi (x)\right)&=\frac{1}{q^N}\sum_{x\in \GN}e\left( \alpha \Phi (x)\right)
\frac{1}{|g|}\sum_{y\in \G_{\ord g}}e\biggl( \frac{y(x-s)}{g} \biggr) \\
&=\frac{1}{|g|}\sum_{y \in \G_{\ord g}}\frac{1}{q^N}\sum_{x\in \GN}
e\biggl( \alpha \Phi (x)+\frac{y(x-s)}{g}\biggr). 
\end{align*}
We observe that the coefficient of $x^k$ in the polynomial $\alpha \Phi (x)+y(x-s)/g$ is 
either $\alpha a_k$ or $\alpha a_k + y/g$, according to whether $k\neq 1$ or $k =1$, and 
in either case this coefficient is irrational. Therefore, it follows from Theorem 
\ref{th:main3} that for any $y\in \G_{\ord g}$, we have
\[
\lim_{N\rightarrow \infty}\frac{1}{q^N}\biggl| \sum_{x\in \GN}
e\biggl( \alpha \Phi (x)+\frac{y(x-s)}{g}\biggr) \biggr| =0,
\]
whence
\[
\lim_{N\rightarrow \infty}\frac{1}{|g|}\sum_{y \in \G_{\ord g}}\frac{1}{q^N}
\biggl| \sum_{x\in \GN}e\biggl( \alpha \Phi (x)+\frac{y(x-s)}{g}\biggr) \biggr| =0.
\]
Combining these relations, we infer that for any irrational $\alpha \in \T$, and for all 
$s\in \F_q[t]$ and monic $g\in \F_q[t]$, one has 
\begin{equation}\label{eq:detector2}
\lim_{N\rightarrow \infty}\frac{1}{q^N}\biggl| \sum_{x\in \GN}
e\left( \alpha \Phi (gx+s)\right)\biggr| =|g|\lim_{N\rightarrow \infty}\frac{1}{q^N}
\biggl| \sum_{\substack{x\in \GN \\ x\equiv s\mmod{g}}}e\left( \alpha \Phi (x)\right) 
\biggr| =0.
\end{equation}
 
\par For any $M \in \Z^+$, let $g_M$ be the product of all of the monic polynomials in 
$\G_M$. We consider a root $s_M\in \Fq[t]$ of $\Phi \mmod{g_M}$, the existence of which 
is guaranteed by our hypotheses concerning $\Phi$. For $\alpha \in \T$, let 
\begin{equation}\label{w19}
T_{M,N}(\alpha)=\frac{1}{q^N}\sum_{x\in \GN}e(\alpha \Phi ( g_M x+s_M)).
\end{equation}
It is useful also to define the associated Fourier coefficients
\[
{\widehat{T_{M,N}}}(h)=\int_\T T_{M,N}(\alpha)e(-\alpha h)\, {\rm d}\alpha .
\]
Then
\[
T_{M,N}(\alpha)=\sum_{h\in \Fq[t]} {\widehat{T_{M,N}}}(h)e(\alpha h).
\]
We now analyze the quantity $T_{M,N}(\alpha)$, dividing our discussion into cases 
according to whether $\alpha$ is rational or irrational.

\noindent {\bf Case 1.} Suppose that $\alpha\in \T$ is irrational. In this case, we find from 
\eqref{eq:detector2} that for any $M\in \Z^+$ and any irrational $\alpha \in \T$, we have
\[
\lim_{N\rightarrow \infty}T_{M,N}(\alpha)=0.
\]

\noindent {\bf Case 2.} Suppose that $\alpha \in \T$ is rational. In this case, we observe 
that a trivial estimate supplies the bound $|T_{M,N}(\alpha)|\le 1$, so that the sequence 
$\left( T_{M,N}(\alpha)\right)_{N\in \Z^+}$ is bounded uniformly in $M$ and $\alpha$. 
Thus, since the set
\[
\{(\alpha ,M):\text{$\alpha \in \T$ is rational and $M\in \Z^+$}\}
\]
is countable, it follows from a diagonalization process that we can extract a subsequence 
$(N_i)_{i=1}^\infty$ of the natural numbers having the property that, for any $M \in \Z^+$ 
and any rational $\alpha \in \T$, the limit 
\[
\lim_{i\rightarrow \infty}T_{M,N_i}(\alpha)
\]
exists. We observe next that $s_M$ is a root of $\Phi\mmod{g_M}$, and hence 
$\Phi(g_M x+s_M)$ is divisible by $g_M$. Consequently, whenever $M$ is large enough that 
$g_M\alpha\in \F_q[t]$, we have $T_{M,N}(\alpha) = 1$.\par

Combining the analyses of the above two cases, we discern that
\[
\lim_{M\rightarrow \infty}\lim_{i\rightarrow \infty}T_{M,N_i}(\alpha)=
\begin{cases}0,&\text{when $\alpha$ is irrational,}\\
1,&\text{when $\alpha$ is rational.}\end{cases}
\]
Now let $\mu$ be a finite non-negative measure on $\T$. By applying the dominated 
convergence theorem twice, we see that 
\[
\lim_{M\rightarrow \infty}\lim_{i\rightarrow \infty}\int_\T T_{M,N_i}(\alpha )
\, {\rm d}\mu (\alpha)=\int_\T \lim_{M\rightarrow \infty}\lim_{i\rightarrow \infty} 
T_{M,N_i}(\alpha )\, {\rm d}\mu (\alpha) 
=\usum_{\substack{\alpha \in \T\\ \text{$\alpha$ rational}}}\mu (\{ \alpha\}),
\]
whence
\begin{equation}\label{w17}
\lim_{M\rightarrow \infty}\lim_{i\rightarrow \infty}\int_\T T_{M,N_i}(\alpha )
\, {\rm d}\mu (\alpha)\geq \mu(\{ 0\}). 
\end{equation}

\par Suppose next that $\widehat{\mu }$ vanishes on $\mathcal{H}$. We note that, on 
recalling the definition \eqref{w16} of $\mathcal{H}$, the definition of $T_{M,N}$ implies 
that we have $\widehat{T_{M,N}}(h)\neq 0$ only if $h\in \mathcal{H}\cup\{0\}$. 
Therefore, we have
\[
\biggl| \int_\T T_{M,N}(\alpha)\, {\rm d}\mu (\alpha)\biggr|=\biggl| \sum_{x\in \Fq[t]}
\widehat{T_{M,N}}(x){\overline{\widehat{\mu}(x)}}\biggr| 
=|\widehat{T_{M,N}}(0)\widehat{\mu}(0)|=|\widehat{T_{M,N}}(0)|\mu(\T ).
\]
On recalling \eqref{w19}, we find that
\[
|\widehat{T_{M,N}}(0)|=\frac{1}{q^N}\text{card}\{ x\in \G_N:\Phi(g_Mx+s_M)=0\}\le 
\frac{\text{deg}(\Phi)}{q^N}.
\]
By working harder, one can confirm that this upper bound $\text{deg}(\Phi)/q^N$ may be 
replaced by $1/q^M$ whenever $M$ is large enough in terms of the coefficients of 
$\Phi(u)$. Hence, we deduce that
\begin{equation}\label{w18}
\biggl| \int_\T T_{M,N}(\alpha)\, {\rm d}\mu (\alpha)\biggr|\leq 
\frac{\text{deg}(\Phi)}{q^N}\mu(\T ).
\end{equation}
Combining the two inequalities \eqref{w17} and \eqref{w18}, we find that $\mu(\{0\})=0$ 
for any finite non-negative measure $\mu $ on $\T$ with $\widehat{\mu}$ vanishing on 
$\mathcal{H}$. Therefore, we deduce from Theorem \ref{theoremw20} that 
$\mathcal{H}$ is van der Corput. 
\end{proof}

\begin{proof} [Proof of Theorem \ref{th:vdc2}]
Suppose that $q=p$ and $\Phi(u)=\sum_{r\in \calK}a_r u^r\in \F_p[t,u]$. Let
\[
\mathcal{H}=\{ \Phi(x): x\in \Fp[t] \} \setminus \{0\}.
\]
Also, let $\calI$ and $S_k(\Phi)$ $(k\in \calI)$ be defined as in (\ref{cali}) and (\ref{eq:s}),
respectively. We have seen in (\ref{esf}) that 
\[
e(\alpha \Phi(x))=e\biggl( \sum_{k\in \calI}S_k(\alpha \Phi )x^k\biggr).
\]
For any $M\in \Z^+$, let $g_M$ be the product of all of the monic polynomials in $\G_M$. 
Then, when $\alpha \in \T$, we put
\[
T_{M,N}(\alpha)=\frac{1}{p^N}\sum_{x \in \GN}e(\alpha \Phi (g_M x))=
\frac{1}{p^N}\sum_{x\in \GN}e\biggl( \sum_{k\in \calI}S_k(\alpha \Phi)(g_M x)^k\biggr).
\] 
If we now define
\[
\calQ =\{ \alpha \in \T: \text{$S_k(\alpha \Phi)$ is irrational for some $k\in \calI$} \},
\]
then we see from (\ref{eq:limit}) that for any $\alpha \in \calQ$, we have
\[
\lim_{N\rightarrow \infty}T_{M,N}(\alpha )=0.
\]
On the other hand, when $\alpha \not \in \calQ$, then $S_k(\alpha \Phi)$ is rational for all 
$k\in \calI$. Since the rational elements $\alpha\in \T$ are countable, the set of all 
polynomials of the form
\[
\sum_{k\in \calI}S_k(\alpha \Phi)y^k\quad (\alpha \not \in \calQ)
\]
is countable. It is worth noting at this point that the set $\T \setminus \calQ$ itself need not 
be countable. Since $|T_{M,N}(\alpha)|\le 1$, it follows via a diagonalization process that 
we can extract a subsequence $(N_i)_{i=1}^\infty$ of natural numbers having the property 
that, for any $M\in \Z^+$ and any $\alpha \not \in \calQ$, the limit
\[
\lim_{i \rightarrow \infty}T_{M,N_i}(\alpha)
\]
exists. Also, by following an argument similar to that applied in Case 2 of the proof of 
Theorem \ref{th:vdc1}, we find that for $M$ sufficiently large, one has 
$T_{M,N}(\alpha )=1$ for any $\alpha \not \in \calQ$. It follows that 
\[
\lim_{M\rightarrow \infty }\lim_{i\rightarrow \infty }T_{M,N_i}(\alpha )=
\begin{cases}0,&\text{when $\alpha \in \calQ$,}\\
1,&\text{when $\alpha \not \in \calQ$.}\end{cases}
\]
We may now argue as in the proof of Theorem \ref{th:vdc1}, mutatis mutandis, to confirm 
that $\mu \left( \{0 \} \right)=0$ for any finite non-negative measure $\mu$ on $\T$ 
satisfying the 
property that $\widehat{\mu }$ vanishes on $\mathcal{H}$. Therefore, we deduce from 
Theorem \ref{theoremw20} that $\mathcal{H}$ is van der Corput. 
\end{proof}

\section{Glasner sets in $\Fq[t]$} \label{sec:glasner}
\subsection{Background and statement of results.} We first introduce some notation and 
nomenclature relevant for the discussion of Glasner sets in $\Fq[t]$. A subset 
$Y\subset \R/\Z$ is called  $\epsilon$-\textit{dense} in $\R/\Z$ if it intersects every interval 
of length $2\epsilon$ in $\R/\Z$. A \textit{dilation} of $Y$ is a set of the form 
$nY=\{ ny:y\in Y\}\subset \R /\Z$ for some $n\in \Z$. In 1979, Glasner \cite{glasner} 
proved that for any infinite subset $Y$ of $\R /\Z$ and any $\epsilon>0$, there exists 
$n\in \Z$ having the property that the dilation $nY$ is $\epsilon$-dense in $\R /\Z$. It 
transpires that the same conclusion can be obtained when one restricts $n$ to be an 
element of a relatively sparse subset of the integers. Motivated by Glasner's theorem, we 
say that a set $\calH \subset \Z$ is \textit{Glasner} if for any infinite subset $Y$ of $\R /\Z$ 
and any $\epsilon>0$, there exists $n\in \calH$ having the property that $nY$ is 
$\epsilon$-dense in $\R /\Z$. In their paper \cite{ap}, Alon and Peres showed that the set 
of primes is Glasner. They also proved that if $\Phi(u)\in \Z[u]$ is a non-constant 
polynomial, then the set $\{ \Phi(n):n\in \Z \}$ is Glasner. By using harmonic analysis, Alon 
and Peres obtained quantitative versions of their results. Thus, for each of the above two 
Glasner sets $\calH$ and any $\epsilon>0$, there exists an $\epsilon$-dense dilation $nY$ 
of $Y$ with $n \in \calH$, provided that the cardinality $|Y|$ of $Y$ is sufficiently large in 
terms of $\epsilon$ and $\calH$. The method and results of Alon and Peres were 
generalized to multi-dimensional tori in \cite{kl} and \cite{bf}.\par

One can define an analog of the notion of a Glasner set in $\Fq[t]$. For $M\in \Z^+$, a 
subset $Y\subset \T$ is called $q^{-M}$-\textit{dense} in $\T$ if it intersects every cylinder 
set $\mathcal{C}$ of radius $q^{-M}$ in $\T$. We call a set $\calH \subset \Fq[t]$ 
\textit{Glasner} if for any infinite subset $Y\subset \T$ and any $M\in \Z^+$, there exists 
$x\in \calH$ having the property that the dilation $xY$ is $q^{-M}$-dense in $\T$. In view 
of the result of Alon and Peres, one may ask if the set of values of a polynomial with 
coefficients in $\Fq[t]$ is Glasner. However, the following examples show that an exact 
analog of the result of Alon and Peres is {\it not} true in general.

\begin{ex}\label{ex:glasner1}
Let $Y$ be the set of all $\alpha \in \T$ with $T(\alpha)=0$, where $T$ is the map defined 
in \eqref{eq:t}. Then $Y$ is infinite (and indeed uncountable). We have seen in Example 
\ref{ex:1} that for any $x\in \Fq[t]$ and $\alpha\in Y$, we have $\res (x^p\alpha )=0$. 
This shows that the set $\{x^p:x \in \Fq[t]\}$ is not Glasner, since for any $x\in \Fq[t]$, the 
set $x^pY$ fails to be $q^{-1}$-dense.
\end{ex}

\begin{ex}\label{ex:glasner2}
Let us assume that $q=p$. Let $Y$ be the set of all $\alpha \in \T$ with 
$T(\alpha)+\alpha=0$. One sees again that $Y$ is infinite (and indeed uncountable). Then 
for any $x \in \Fq[t]$, we have $\res ((x^p+x)\alpha) = \res ((T(\alpha)+\alpha)x)=0$. 
This shows that the set $\{ x^p+x:x \in \Fq[t]\}$ is not Glasner, since for any $x\in \Fq[t]$, 
the set $(x^p+x)Y$ fails to be $q^{-1}$-dense.
\end{ex}

One could formulate a conjecture similar to Conjecture \ref{conj} asserting that Examples 
\ref{ex:glasner1} and \ref{ex:glasner2} encapsulate all the obstructions preventing a 
polynomial sequence in $\Fq[t]$ from being Glasner. We have some preliminary ideas that 
might establish such a conjecture, and this is a subject to which we intend to return on a 
future occasion. For now we note that such a conjecture would follow from Conjecture 
\ref{conj}. Moreover, partial progress is made possible by making use of Theorem 
\ref{th:main3}. Here, once again, we recall the definition \eqref{eq:k-star} of the set of 
exponents $\calK^*$.

\begin{thm}\label{theoremw21}
Let $\calK$ be a finite set of positive integers, suppose that $a_r\in \Fq[t]$ for 
$r\in \calK\cup\{0\}$, and define
\[
\Phi(u)=\sum_{r\in \calK\cup\{0\}}a_r u^{r}.
\]
Suppose further that $a_k\neq 0$ for some $k\in \calK^*$ with $k>1$. Then the set 
$\{ \Phi(x):x\in \Fq[t]\}$ is Glasner.
\end{thm}

Notice the extra requirement $k>1$ in Theorem \ref{theoremw21}, a condition absent from 
the hypotheses of Theorem \ref{th:sarkozy}. By adapting the harmonic-analytic approach of 
Alon and Peres described in \cite {ap}, we prove the following quantitative version of 
Theorem \ref{th:glasner} analogous to the bound of Alon and Peres obtained in 
\cite[Theorem 6.3]{ap}.

\begin{thm}\label{th:glasner2}
Let $\calK$ be a finite set of positive integers, suppose that $a_r\in \Fq[t]$ for 
$r\in \calK\cup\{0\}$, and define
\[
\Phi(u)=\sum_{r\in \calK\cup\{0\}}a_r u^{r}.
\]
Suppose further that $a_k\neq 0$ for some $k\in \calK^*$ with $k>1$. Then there exists a 
positive constant $C$, depending on $\Phi$, such that whenever $M>0$ and 
$|Y|\ge q^{CM}$, there is a dilation of the form $\Phi(x)Y$ of $Y$ that is $q^{-M}$-dense.
\end{thm}

We remark that, as a direct consequence of Theorem \ref{th:glasner}, the set of values of 
$\Phi$ is Glasner whenever $\deg \Phi>1$ and $(\deg \Phi, p)=1$. Also, in view of 
Proposition \ref{prop:varmain3}, the condition $a_k\neq 0$ for some $k\in \calK^*$ can be 
relaxed to the constraint that $a_k\neq 0$ for some $k \in \tK$, where $\tK$ is defined as 
in (\ref{eq:ktilde}). 

\subsection{Proof of Theorem \ref{th:glasner2}}
We first derive the following cheap consequence of Theorem \ref{th:main2}. It is analogous 
to Hua's classical bound on complete exponential sums with polynomial argument over the 
integers, a version of which could certainly be derived in the setting of $\Fq[t]$. Whilst the 
latter would deliver stronger conclusions than those we obtain below, the extra effort 
involved has no impact on the application that we have in mind.    

\begin{lem}\label{lem:hua}
Let $\calK$ be a finite set of positive integers, suppose that $a_r\in \Fq[t]$ for 
$r\in \calK\cup\{0\}$, and define
\[
\Phi(u)=\sum_{r\in \calK\cup\{0\}}a_r u^{r}.
\]
Suppose further that $a_k\neq 0$ for some $k\in \calK^*$ with $k>1$. Then there exists a 
constant $C_k>1$, depending only on $k$, $\calK$ and $q$, such that for any monic 
$g\in \Fq[t]$ and any $\epsilon>0$, we have
\begin{equation}\label{eq:hua}
\biggl| \sum_{x\in \G_{\ord_g}}e\biggl(\frac{\Phi(x)}{g} \biggr) \biggr| 
\ll_{\calK ,\epsilon ,q}|(g,a_k)|^{1/C_k}|g|^{1-1/C_k+\epsilon }.
\end{equation}
\end{lem}

\begin{proof} We fix the positive constants $c_k$ and $C_k$, depending at most on $k$, 
$\calK$ and $q$, in accordance with the conclusion of Theorem \ref{th:main2}. Write 
$N=\ord g$ and $M=\ord (g,a_k)$, and put
\begin{equation}\label{w22}
\eta =\min \{ c_kN, (1/C_k-\epsilon)N-M/C_k\}.
\end{equation}
On observing 
that the bound \eqref{eq:hua} is trivial when $\eta \le 0$, we see that there is no loss of 
generality in assuming henceforth that $\eta>0$. We may also suppose that $N$ is 
sufficiently large in terms of $\calK$, $\epsilon$ and $q$. Suppose, by way of deriving a 
contradiction, that
\[
\biggl| \sum_{x\in \G_{N}}e\biggl( \frac{\Phi(x)}{g}\biggr) \biggr| \geq q^{N-\eta }.
\]
Then we infer from Theorem \ref{th:main2} that there exist $b\in \Fq[t]$ and monic 
$h\in \Fq[t]$ such that
\begin{equation}\label{eq:huaproof1}
\ord \left( h\frac{a_k}{g}-b\right) <-kN+\epsilon N+C_k\eta \quad \textup{and} \quad 
\ord h\le \epsilon N+C_k\eta.
\end{equation}

\par We see from \eqref{w22} that $M+C_k\eta\le (1-C_k\epsilon)N$. It therefore follows 
from \eqref{eq:huaproof1} that 
\[
\ord (g,a_kh)\leq M+\ord h\leq M+\epsilon N+C_k\eta \le (1+\epsilon -C_k\epsilon )N<N.
\]
Since $\ord g=N$, we deduce that $g$ does not divide $(g, a_k h)$. Consequently, the 
fraction $ha_k/g$ has a reduced form with denominator $g/(g,a_kh)$ having order at least 
$1$. Thus, we have
\begin{equation}\label{eq:huaproof3}
\ord \left( h\frac{a_k}{g}-b\right) \geq \ord \left( \frac{1}{g/(g,a_kh)}\right) =
\ord (g,a_k h)-\ord g\geq M - N. 
\end{equation} 
Combining (\ref{eq:huaproof1}) and (\ref{eq:huaproof3}), we obtain the bound
\[
M-N\leq -kN+\epsilon N+C_k\eta \le -kN+\epsilon N+(1-C_k\epsilon)N-M.
\]
Since $k>1$, we arrive at a contradiction. We are therefore forced to conclude that 
$\eta\le 0$, a scenario in which the conclusion of the lemma follows, as we have already 
observed.
\end{proof}

As we have already noted, one may prove the bound \eqref{eq:hua} by more classical 
methods. Thus, with additional effort it would be possible to establish a version of Lemma 
\ref{lem:hua} with $C_k=\text{deg}(\Phi)$. We also need an analog of 
\cite[Proposition 1.3]{ap}, the statement of which requires that we introduce some 
additional notation. Consider a set $Y=\{y_1,\ldots ,y_k\}\subset \T$. For each 
$g\in \Fq[t]\setminus\{0\}$, we denote by $h_g=h_g(Y)$ the number of pairs $(i,j)$ with 
$1\leq i,j\leq k$ and $i\neq j$ satisfying $g(y_i-y_j)\in \Fq[t]$. Finally, we define 
$H_L=H_L(Y)$ by putting
\[
H_L(Y)=\sum_{g\in \GL\setminus\{0\}}h_g(Y).
\]

\begin{lem}\label{lem:h1}
Let $Y=\{y_1,\ldots ,y_k\}$ be a set of $k$ distinct elements in $\T$. Then for each 
non-negative integer $L$, one has $H_L(Y)\leq kq^{2L}$.
\end{lem}

\begin{proof}
For each index $i$ with $1\leq i\leq k$ and $g\in \GL\setminus\{0\}$, the number of indices 
$j$ for which $g(y_i-y_j)\in \Fq[t]$ is at most $|g|\leq q^L$. Thus, we deduce that
\[
H_L(Y)\le \sum_{1\le i\le k}\sum_{g\in \GL\setminus \{0\}}q^L\le kq^{2L},
\]
and the proof of the lemma is complete.
\end{proof}

\begin{proof}[Proof of Theorem \ref{th:glasner2}] 
We prove Theorem \ref{th:glasner2} by establishing the contrapositive. Suppose then that a 
set of $k$ distinct elements $Y=\{y_1,\ldots, y_k\}\subset \T$ has the property that 
$\Phi(x)Y$ is not $q^{-M}$-dense for any $x\in \Fq[t]$. We seek to derive an upper bound 
for $k$ of the shape $k<q^{CM}$, with $C$ a suitable positive constant depending on 
$\Phi$.\par

Consider any element $x\in \Fq[t]$. We may suppose that $\Phi(x)Y$ is not 
$q^{-M}$-dense in $\T$, and hence there exists $\xi_x\in \T$ having the property that all 
elements of $\Phi(x)Y$ lie outside of the cylinder set $\{ \xi \in \T: |\xi -\xi_x|<q^{-M} \}$. 
Thus, for all $1\leq i\leq k$, we have
\[
\ord \left\{ \Phi(x)y_i - \xi_x \right\} \geq -M.
\]
In view of \eqref{eq:orthogonal1}, we see that for each $x\in \Fq[t]$ and index $i$, one 
has 
\[
\sum_{z\in \GM}e\left( z\left( \Phi(x)y_i-\xi_x\right) \right)=0. 
\]
Consequently, isolating the term $z=0$ in each sum, we deduce that for each positive 
integer $N$ one has the relation
\[
\sum_{x\in \GN}\sum_{i=1}^k\sum_{z\in \GM\setminus \{0\}}
e\left( z\left( \Phi(x)y_i-\xi_x\right) \right) =-kq^N. 
\]
Interchanging the innermost summations and applying Cauchy's inequality, we therefore 
obtain the relation
\begin{align*}
k^2 q^{2N}&\leq q^{N+M}\sum_{x\in \GN}\sum_{z\in \GM\setminus\{0\}} 
\biggl| \sum_{i=1}^k e\left( z\left( \Phi(x)y_i-\xi_x\right) \right) \biggr|^2\\
&=q^{N+M}\sum_{x\in \GN}\sum_{z\in \GM\setminus\{0\}}\sum_{i=1}^k\sum_{j=1}^k
e\left( z\Phi(x)(y_i-y_j)\right) .
\end{align*}
Therefore, again interchanging orders of summation, we find that
\begin{align}
k^2&\le q^M\sum_{z \in \GM\setminus\{0\}}\sum_{i=1}^k \sum_{j=1}^k \frac{1}{q^N}
\sum_{x\in \GN}e\left( z\Phi(x)(y_i-y_j)\right) \notag \\
&\le q^{2M}\max_{z \in \GM\setminus\{0\}}\sum_{i=1}^k \sum_{j=1}^k 
\Theta(z;y_i-y_j),\label{eq:limit2}
\end{align}
where
\begin{equation}\label{w23}
\Theta (z;u)=\limsup_{N\rightarrow \infty}\biggl| \frac{1}{q^N}\sum_{x\in \GN}
e\left( z\Phi(x)u\right) \biggr| .
\end{equation}

\par We now analyse the limit $\Theta(z;y_i-y_j)$ when $z\in \GM\setminus \{0\}$, with the 
result depending on whether or not $y_i - y_j$ is rational.

\noindent \textbf{Case 1.} Suppose that $i=j$. Then we find from \eqref{w23} that 
$\Theta(z;y_i-y_j)=\Theta(z;0)=1$.

\noindent \textbf{Case 2.} Suppose that $y_i-y_j$ is irrational. In this scenario, when 
$z\in \GM\setminus \{0\}$, we find that $z(y_i-y_j)$ is also irrational, and hence it follows 
from Theorem \ref{th:main3} that $\Theta(z;y_i-y_j)=0$.

\noindent \textbf{Case 3.} Suppose that $y_i-y_j$ is a non-zero rational. In these 
circumstances, we write $y_i-y_j=a/g$ as a reduced fraction with $a\in \Fq[t]$ and monic 
$g\in \Fq[t]$. Given $z\in \GM\setminus\{0\}$, we may in turn write 
$z(y_i-y_j)=a'/g'$ as a reduced fraction with $g'=g/(z,g)$ and $a'=az/(z,g)$. In particular, 
therefore, we have $|g'|\geq |g|/q^M$. We now recall \eqref{w23} and appeal to Lemma 
\ref{lem:hua}. Thus, there exists a constant $C_k>1$, depending only on $k$, $\calK$ and 
$q$, such that
\[
\Theta(z;y_i-y_j)=\frac{1}{|g'|}\biggl| \sum_{x\in \G_{\ord g'}}
e\biggl( \frac{a'\Phi(x)}{g'}\biggr) \biggr| \ll_{\calK} |g'|^{- 1/(2C_k)}|(g',a'a_k)|^{1/C_k}.
\]
On noting that $(g',a')=1$, we deduce that
\begin{equation}\label{eq:huaapp}
\Theta(z;y_i-y_j)\ll_\calK |g'|^{-1/(2C_k)}|a_k|^{1/C_k} \ll_\Phi |g|^{-1/(2C_k)}q^M. 
\end{equation}

\par For each monic $g\in \Fq[t]\setminus\{0\}$, denote by $\widetilde{h}_g$ the number 
of pairs $(i,j)$ with $1\leq i,j\leq k$ and $i\neq j$ satisfying the condition that $y_i-y_j$ may 
be written as a reduced fraction with denominator $g$. Then it follows from 
\eqref{eq:limit2} via \eqref{eq:huaapp} and the above analysis dividing into three cases 
that we have the estimate
\begin{equation}\label{eq:htilde}
k^2\ll_{\Phi} kq^{2M}+q^{3M}
\sum_{\substack{g\in \Fq[t] \\ \text{$g$ monic}}}|g|^{-1/(2C_k)}
\widetilde{h}_g.
\end{equation}

\par Next we estimate the right hand side of (\ref{eq:htilde}) using Lemma \ref{lem:h1}. 
For any $L\in \Z^+$, let 
\[
\widetilde{H}_L=\sum_{\substack{g\in \GL \\ \text{$g$ monic}}}\widetilde{h}_g.
\] 
On noting that $\widetilde{H}_1=0$, we find by partial summation that
\begin{align}
\sum_{\substack{g\in \Fq[t] \\ \text{$g$ monic}}}|g|^{-1/(2C_k)}
\widetilde{h}_g&=\sum_{L=1}^\infty q^{-L/(2C_k)}\bigl( 
\widetilde{H}_{L+1}-\widetilde{H}_L\bigr) \notag \\
&=\sum_{L=2}^\infty \widetilde{H}_L\bigl( q^{-(L-1)/(2C_k)}-q^{-L/(2C_k)}\bigr) .
\label{w24}
\end{align}
For any non-negative integer $L$, we have the trivial estimate $\widetilde{H}_L\leq k^2$. 
Meanwhile, as a consequence of Lemma \ref{lem:h1}, we have 
$\widetilde{H}_L\leq H_L\leq k q^{2L}$. Write $L_0=\left\lfloor (\log_q k)/2\right\rfloor $. 
Then
\begin{align*}
\sum_{L=2}^{L_0}\widetilde{H}_L\bigl( q^{-(L-1)/(2C_k)}-q^{-L/(2C_k)}\bigr) &\le 
k\sum_{L=2}^{L_0}q^{2L}\bigl( q^{-(L-1)/(2C_k)}-q^{-L/(2C_k)}\bigr) \\
&\le 2kq^{1+L_0\left( 2-1/(2C_k)\right)}
\end{align*}
and
\begin{align*}
\sum_{L=L_0+1}^\infty \widetilde{H}_L\bigl( q^{-(L-1)/(2C_k)}-q^{-L/(2C_k)}\bigr) &\le 
k^2\sum_{L=L_0+1}^\infty \bigl( q^{-(L-1)/(2C_k)}-q^{-L/(2C_k)}\bigr) \\
&\le k^2q^{-L_0/(2C_k)}.
\end{align*}
On recalling that $L_0=\left\lfloor (\log_q k)/2\right\rfloor $ and substituting these bounds 
into \eqref{w24}, we see that
\[
\sum_{\substack{g\in \Fq[t] \\ \text{$g$ monic}}}|g|^{-1/(2C_k)}
\widetilde{h}_g\le 3qk^{2-1/(4C_k)}.
\]
Equipped with this estimate, the relation \eqref{eq:htilde} now yields the bound 
\[
k^2\ll_\Phi kq^{2M} + q^{3M+1}k^{2-1/(4C_k)},
\]
and thus $|Y|=k\ll_\Phi q^{4C_k (3M+1)}$. In view of our opening discussion, this 
completes the proof of Theorem \ref{th:glasner2}.
\end{proof}

\end{document}